\documentclass{amsart}
\usepackage{color}

\usepackage{amsmath}
\usepackage{amssymb}
\usepackage{mathrsfs}
\usepackage{mathtools}

\newtheorem{theorem}{Theorem}[section]
\newtheorem{claim}[theorem]{Claim}

\newtheorem{majorclaim}[theorem]{Major Claim}

\newtheorem{conclusion}[theorem]{Conclusion}

\theoremstyle{definition}
\newtheorem{definition}[theorem]{Definition}

\newtheorem{discussion}[theorem]{Discussion}

\theoremstyle{remark}
\newtheorem{remark}[theorem]{Remark}

\newtheorem{notation}[theorem]{Notation}

\newcommand{\redq}[1]{{\color{red} #1}}

\newcommand{\hit}{{\rm ht}}

\newcommand{\fl}{{\rm fl}}
\newcommand{\qf}{{\rm qf}}

\newcommand{\otp}{{\rm otp}}

\newcommand{\lev}{{\rm lev}}
\newcommand{\sub}{{\rm sub}}  

\newcommand{\eend}{{\rm end}}

\newcommand{\Lev}{{\rm Lev}}

\newcommand{\eseq}{{\rm eseq}}

\newcommand{\suc}{{\rm suc}}

\newcommand{\level}{{\rm level}}
\newcommand{\seq}{{\rm seq}}
\newcommand{\seqq}{{\rm seq}}

\newcommand{\GEM}{{\rm GEM}}

\newcommand{\ZF}{{\rm ZF}}
\newcommand{\AC}{{\rm AC}}

\newcommand{\ZFC}{{\rm ZFC}}

\newcommand{\tp}{{\rm tp}}

\newcommand{\dom}{{\rm dom}}

\newcommand{\Cohen}{{\rm Cohen}}

\newcommand{\lex}{{\rm lex}}

\newcommand{\bfa}{{\mathbf a}}
\newcommand{\bfH}{{\mathbf H}}

\newcommand{\bff}{{\mathbf f}}

\newcommand{\bfG}{{\mathbf G}}

\newcommand{\bfb}{{\mathbf b}}
\newcommand{\bfm}{{\mathbf m}}

\newcommand{\bfc}{{\mathbf c}}

\newcommand{\bfV}{{\mathbf V}}
\newcommand{\bfT}{{\mathbf T}}

\newcommand{\bfB}{{\mathbf B}}

\newcommand{\EC}{{\rm EC}}

\newcommand{\pos}{{\rm pos}}
\newcommand{\Rang}{{\rm Rang}}

\newcommand{\rang}{{\rm rang}}

\newcommand{\cl}{{\rm  cl}}
\newcommand{\stp}{{\rm sim-tp}}

\newcommand{\lqq }{{``}}

\newcommand{\exiota}{{  }}

\newcommand{\icr}{{\rm incr}}

\newcommand{\rest}{{\restriction}}

\newcommand{\then}{{\underline{then}}}
\newcommand{\when}{{\underline{when}}}

\newcommand{\Iff}{{\underline{iff}}}

\newcommand{\cB}{{\mathscr B}}

\newcommand{\gB}{{\mathfrak B}}

\newcommand{\varp}{{\varepsilon}}

\newcommand{\cH}{{\mathscr H}}

\newcommand{\cF}{{\mathscr F}}

\newcommand{\bbL}{{\mathbb L}}
\newcommand{\bbN}{{\mathbb N}}

\newcommand{\bbP}{{\mathbb P}}

\newcommand{\cT}{{\mathscr T}}

\newcommand{\cU}{{\mathscr U}}

\DeclareMathOperator{\AP}{AP}

\DeclareMathOperator{\End}{end}

\newcount\skewfactor
\def\mathunderaccent#1#2 {\let\theaccent#1\skewfactor#2
\mathpalette\putaccentunder}
\def\putaccentunder#1#2{\oalign{$#1#2$\crcr\hidewidth
\vbox to.2ex{\hbox{$#1\skew\skewfactor\theaccent{}$}\vss}\hidewidth}}
\def\name{\mathunderaccent\tilde-3 }

\newbox\noforkbox \newdimen\forklinewidth
\setbox0\hbox{$\textstyle\bigcup$}
\setbox1\hbox to \wd0{\hfil\vrule width \forklinewidth depth \dp0
                        height \ht0 \hfil}
\setbox\noforkbox\hbox{\box1\box0\relax}
\def\unionstick{\mathop{\copy\noforkbox}\limits}
\def\nonfork#1#2_#3{#1\unionstick_{\textstyle #3}#2}
\def\nonforkin#1#2_#3^#4{#1\unionstick_{\textstyle #3}^{\textstyle
    #4}#2}
\setbox0\hbox{$\textstyle\bigcup$}
\setbox1\hbox to \wd0{\hfil{\sl /\/}\hfil}
\setbox2\hbox to \wd0{\hfil\vrule height \ht0 depth \dp0 width
                                \forklinewidth\hfil}
\newbox\doesforkbox
\setbox\doesforkbox\hbox{\box1\box0\relax}
\def\nunionstick{\mathop{\copy\doesforkbox}\limits}

\def\fork#1#2_#3{#1\nunionstick_{\textstyle #3}#2}
\def\forkin#1#2_#3^#4{#1\nunionstick_{\textstyle #3}^{\textstyle
    #4}#2}

\newcommand{\stickT}{%
\setbox255=\hbox{\raise1ex\hbox{$\hspace{0.2pt}\,\bullet\,$}}
\mathord{\rlap{\hbox to\wd255{\hss\hbox{$|$}\hss}}
\box255}
}
\newcommand{\stickS}{%
\setbox255=\hbox{\raise0.6ex\hbox{$\scriptstyle\bullet$}}
\mathord{\rlap{\hbox to\wd255{\hss\hbox{$\scriptstyle|$}\hss}}
\box255}
}

\newenvironment{PROOF}[2][\proofname.]
   {\begin{proof}[#1]\renewcommand{\qedsymbol}{\textsquare$_{\rm #2}$}}
   {\end{proof}}

\numberwithin{equation}{section}

\usepackage[hidelinks]{hyperref}

\begin{document}
\makeatletter\def\shfiuwefootnote{\gdef\@thefnmark{}\@footnotetext}\makeatother\shfiuwefootnote{Version 2026-01-01. See \url{https://shelah.logic.at/papers/1176/} for possible updates.}

\title{Partition theorems for expanded trees \\
1176}

\author{Saharon Shelah}

\address{Einstein Institute of Mathematics\\
Edmond J. Safra Campus, Givat Ram\\
The Hebrew University of Jerusalem\\
Jerusalem, 9190401, Israel\\
 and \\
 Department of Mathematics\\
 Hill Center - Busch Campus \\
 Rutgers, The State University of New Jersey \\
 110 Frelinghuysen Road \\
 Piscataway, NJ 08854-8019 USA}
 
\email{shelah@math.huji.ac.il}

\urladdr{http://shelah.logic.at}

\thanks{The author thanks Alice Leonhardt for the beautiful typing up to 2019. In later versions, the author thanks typing services generously funded by 
Craig Falls 
and we thank the typist   
for the careful and beautiful typing. 
We thank the ISF (Israel Science Foundation) grant 1838(19)  
(2019-1023) and grant 2320/23 (2023-2027)
for partially supporting this research. First typed November 14, 2018.    References like \cite[Th0.2=Ly5]{Sh:950} means the label of Th.0.2 is y5. The reader should note that the version in my website is usually more updated than the one in the mathematical archive.  Publication number 1176 in the author list of publications}


\date{December 29, 2025}

\begin{abstract}
    We look for partition theorems for large subtrees for suitable uncountable trees and colourings parallely to the statement $  \lambda \rightarrow (\mu )^n_\kappa $ such that possibly $ \lambda > \mu$.
    
    \noindent
    We concentrate on sub-trees of ${}^{\kappa \geq}2$ expanded by a well-ordering of each level. However, in the  embedding 
    the equality of levels is 
    preserved. 
    The gain is that  
    we get consistency results without large cardinals. 
    
    \noindent
    An intention is to  apply the results 
    to model theoretic problems. 
\end{abstract}

\maketitle

\setcounter{section}{-1}

\centerline{Annotated Content}
\bigskip

\noindent
\S0 \quad Introduction, pg. \pageref{0}
\bigskip

\noindent
\S(0A) \quad Background, pg. \pageref{0A}
\bigskip

\noindent
\S(0B) \quad Preliminaries, (label z) pg. \pageref{0B}
\bigskip

\noindent
\S1 \quad Partition Theorems, (label a), pg.\pageref{1B}

\noindent
\S1A \quad The definitions, pg.\pageref{1A}

\begin{enumerate}
    \item[${{}}$]   [We consider here  partition theorems for trees. This is used in \S1B.
    
    In~\ref{a28} define $\bfT$.  
    \newline
    Definition~\ref{a40} states the partition relation.]
\end{enumerate}
\bigskip   

\noindent
\S1B \quad  Forcing in $\ZFC$, (label b) pg. \pageref{1B}

\begin{enumerate}
    \item[${{}}$]  [We work on a partition theorem for trees, for the main case here,  we get consistency without the large cardinal. Naturally the price is having to vary the size of the cardinals (parallel to the  Erd\"os-Rado  theorem).]
\end{enumerate} 










\section{Introduction}\label{0}

\subsection{Background and Results}\label{0A}\

We continue two lines of research. One is set theoretic:  pure partition relations on trees and the other is model theoretic:  Hanf numbers
and non-deniability of well ordering, in particular related to $ \omega _1 $. This is related to the existence of $\GEM$ (generalized Eherenfuecht-Mostowski)  for suitable templates (see \cite{Sh:E59}), and applications to descriptive set theory.
 
Halpern-Levy \cite{HaLe67} had proved a milestone theorem on independence of versions of the axiom of choice: in $\ZF,\AC$ is strictly stronger
than the maximal prime ideal theorem (i.e. every Boolean algebra has a maximal ideal).

This work isolated a partition theorem\footnote{Using  not splitting to 2 but other finite splitting makes a minor difference; similarly here.} on the tree ${}^{\omega >}2$, necessary  for the proof.  This partition theorem was subsequently proved by Halpern-Lauchli \cite{HaLa66} and was a major and early theorem in Ramsey theory,  (so the proof above relies on it). 

See more in Laver \cite{Lv71}, \cite{Lv73} and \cite[AP,\S2]{Sh:a} and Milliken
\cite{Mil79}, \cite{Mil81}.

The \cite{HaLa66} proof uses induction, later Harrington found a different proof using forcing:  adding many Cohen reals and a name of a (non-principal) ultrafilter on $\bbN$.  Earlier, (on adding many reals and  a partition theorem) see Silver's proof  of $\Pi^1_1$-equivalence relations, in \cite{Sil80}.  
 
Now \cite[\S4]{Sh:288} turns to uncountable trees, i.e. for some $\kappa > \aleph_0$, we consider  trees $\cT$  which are sub-trees of $({}^{\kappa >}2,\triangleleft)$, such that (as in \cite{HaLa66}) for every level $\varp < \kappa$, either $(\forall \eta \in \cT \cap {}^\varp 2)(\eta \char 94 \langle 0 \rangle,\eta \char 94 \langle 1 \rangle \in \cT)$ or $(\forall \eta \in
\cT \cap {}^\varp 2)(\exists !   \iota < 2)[\eta \char 94 \langle \iota\rangle \in \cT]$;  
(and of course the first  occurs unboundedly often). But a new point is that we have to use a well ordering of $\cT \cap {}^\varp 2$ for $\varp < \kappa$. 

Naturally we add   \lqq is closed enough (that is under unions of increasing sequences
 of length $ < \kappa $)". Also colouring with
infinite number of colours, the proof uses ``measurable $\kappa$  which remains so when we add $\lambda$ many $\kappa$-Cohens for appropriate
$\lambda$"; it generalizes Harrington's proof.
This was continued  in several works, see Dobrinen-Hathaway \cite{DoHa17} and 
references there.  

We are here mainly interested in a weaker version
which is  enough for the  model theoretic applications we have in mind,  we start with  a large tree and get one of smaller cardinality,
in a sense this is solving the ``equations'' $X$  /(Erd\"os-Rado theorem) = \cite{Sh:288}/ (the partition relation of a weakly compact cardinal) = \cite{HaLa66} / (Ramsey theorem). On other consistent partition relation see Boney-Shelah \cite{Sh:F1822}, in preparation.

Turning to model theory see \cite{Sh:37}, \cite{Sh:49} and Dzamonja-Shelah \cite{Sh:692} where such indiscernibility is considered in a model theoretic context.  

A central direction in model theory in the sixties were two cardinal theorems.  For infinite cardinals $ \mu > \lambda,$ let $ K_{\mu, \lambda }$ be the class of models $ M $ such that $ M $ is of cardinality $ \mu $ and $ P^M $ of cardinality  $ \lambda $. The main problems were transfer, compactness and completeness. For connection to partition theorems, Morley's   proof of  \cite{Va65}, the  Vaught far apart two cardinal theorem  used  Erd\"os-Rado theorem;  generally see \cite{Sh:8}, \cite{Sh:E17}, \cite{Sh:74} and the survey \cite{Sh:85}.  Jensen's celebrated gap $ n $ two cardinal theorem solves  
those problems for e.g  $( \aleph _n , {\aleph_0} )$  \when \, $ \mathbf{V} = \mathbf{L} $. But can we get a nice picture in different universes?

Note that  by \cite{Sh:289}, \cite{Sh:E59}, consistently we have $\GEM$ (generalized Eherenfuecht-Mostowski) models for ordered graphs  
as index models, even omitting types. 

On a different direction  
D.Ulrich has asked me on $(*)_n$ below (and told me it has descriptive set theoretic consequences, see \cite{Sh:1141}). We intend to prove (in the sequel \cite{Sh:F2060}) that  for $ n < \omega$:  

\begin{enumerate} 
    \item[$(*)_n$] consistently 
    
    \begin{enumerate} 
        \item[(a)] if $\psi \in \mathbb{L} _{{\aleph_1}, {\aleph_0} }$ has a model $ M $ of cardinality $ \beth _{n+1}$ with 
        $ (P^M, < ^M )$ having order type $ \omega _1 $  \then \, $ \psi $ has a model $ N $ of cardinality $ \beth _{n + 1 }$
        and $ (P^N, < ^N ) $ is not well ordered, 
        
        \item[(b)] moreover, it is enough that $ M $ will have cardinality $\aleph _ \delta ,  \delta \ge  \beth _n ^{++}$,  
        
        \item[(c)] of course,  preferably not using large cardinals. 
    \end{enumerate} 
\end{enumerate} 

This requires consistency of many cases of partition relations on trees and more complicated structures, analysing GEM models. Much earlier we  have intended (mentioned in \cite[1.15]{Sh:702})
to prove the parallel for first order logic; and $ (\beth_n, {\aleph_0}) $, using  (many Cohen  
indestructible)  measurables  $ \kappa _1 < \cdots < \kappa _n $ as in \cite[\S4]{Sh:288} and forcing by blowing $ 2 ^ {\aleph_0} $  to $ \kappa _1$, 
$ 2^ {\kappa _1 }$ to $ \kappa _2 $ etc  relying
on \cite{Sh:288}; but have not carried out that.  

In preparation are  also 
solutions to the two cardinal problems above and 

\begin{enumerate}
    \item[$(*)$]  
    
    \begin{enumerate}
        \item[(a)] $\alpha_\bullet < \omega_1,\beth_{\alpha +1} = (\beth_\alpha)^{+ \omega_1 + 1}$ for $\alpha < \alpha_\bullet$ and well ordering of $\omega_1$ is not definable in $\{\EC_\psi(\beth_{\alpha_\bullet })\colon\psi \in \bbL_{\aleph_1,\aleph_0}\}$ or at least,
        
        \item[(b)]  as above  but for $\beth_{\alpha +1} = \aleph_{\beth^{++}_\alpha},$
        
        \item[(c)] parallel results replacing $ {\aleph_0} $ by $ \mu.$
    \end{enumerate}
\end{enumerate}

Contrary to the a priori expectation no large cardinal is used. 

In a  sequel \cite{Sh:F2064}  we intend also to deal with other partition relations  and with weakly compact cardinals.  

We thank Shimoni Garti  and Mark Po\'{o}r for many helpful comments. 

\subsection{Preliminaries}\label{0B}\

\begin{definition}\label{z7}
    If $\mu = \mu^{< \kappa}$ then ``\emph{for a $(\mu,\kappa)$-club of $u \subseteq X$ we have $\varphi(u)$}" means that: for  some $ \chi $ such that $\mu, X  
    \in {\mathscr H} ( \chi ) $ 
    and  e.g. $ \beth _3 (\mu + | X  
    |) < \chi $  and some $x \in \cH(\chi),$ if $x \in \cB \prec (\cH(\chi),\in), \, \|\cB\| = \mu, \, [\cB]^{< \kappa} \subseteq \cB$ and   $ \mu + 1 \subseteq  \cB$, \then \, the set $u = \cB \cap X$ satisfies $\varphi(u)$; there are other variants.
\end{definition}

\begin{definition}\label{z19}
    For $\kappa$ regular (usually $\kappa = \kappa^{<\kappa}$) and an ordinal $\gamma,$ the forcing $\bbP =$ Cohen$(\kappa,\gamma)$ of adding  $\gamma$   many $\kappa$-Cohen reals  is defined as follows:
    
    \begin{enumerate}
        \item[(A)]  $p \in \bbP$ \Iff:
        
        \begin{enumerate}
            \item[(a)]  $p$ is a function with domain from $[\gamma ]^{< \kappa}$,
            
            \item[(b)]  if $\alpha \in \dom(p)$ then $p(\alpha) \in {}^{\kappa >}2$,
        \end{enumerate}
        
        \item[(B)]  $\bbP \models p \le q$ \Iff:
        
        \begin{enumerate}
            \item[(a)]  $p,q \in \bbP$, 
            
            \item[(b)]  $\dom(p) \subseteq \dom(q)$,  
            
            \item[(c)]  if $\alpha \in \dom(p)$ then $p(\alpha) \trianglelefteq q(\alpha).$
        \end{enumerate}
        
        \item[(C)]  for $\alpha < \gamma $ let $\name\eta_\alpha = \bigcup \{p(\alpha) \colon p \in \name{\bfG}_{\bbP}$ satisfies $\alpha \in \dom(p)\}$, so $\Vdash_{\bbP} ``\name\eta_\alpha \in {}^\kappa 2"$, 
        
        \item[(D)]  for $u \subseteq \gamma$
        let $\bbP_u \coloneqq \{p \in \bbP \colon \dom(p) \subseteq u\},$ so $\bbP_u \lessdot \bbP$ and $\bar\eta_u = \langle \name\eta_\alpha:\alpha \in u\rangle$ is generic for $\bbP_u$.
    \end{enumerate}
\end{definition}

\begin{notation}\label{z22}
    1) We denote infinite cardinals by 	$ \kappa, \lambda, \mu, \chi, \theta, \partial $, and $ \sigma $ denotes a possibly finite cardinal.
	
    2) We denote ordinals by $ \alpha, \beta, \gamma, \delta, \varepsilon, \zeta, \xi $ and sometimes $ i,j$.
    	
    3) We denote natural numbers by $ k, {\ell}, m, n$ and sometimes $ i,j$.	
    
    4) Instead   of 
    e.g. $a_{i}$ we may write $a[i],$ particularly in sub-script; also $\kappa(+)$ means $\kappa^{+}.$  

    5) Let $ h[u] ;= \{ h(x): x \in u \}. $  
\end{notation} 

\begin{notation}\label{z25}
    Concerning Definition \ref{a28}: 
    
    1) Let $\cT$ denote members of $\bfT_{\fl}$ or of $\bfT_{\mathrm{wk}}$, writing  $\bfT$ mean it can be either.

    2) Similarly about $\cT_{1} \subseteq_{\mathrm{wk}} \cT_{2}$ or $\cT_{1} \subseteq_{\fl} \cT_{2}$ (or embedability).

    3) Also, in Definition \ref{a40}, we have either $\to_{\fl}$ or $\to_{\mathrm{wk}}$.
\end{notation}


\section{Partition Theorems}\label{1}

\subsection{The definitions}\label{1A}\

Here, we consider partitions on trees. For uncountable trees, we find the need to consider a well-ordering of each level, still preserving equality of level.  We may consider embeddings where equality of levels is not preserved, see Dzamonja-Shelah \cite{Sh:692}  (in the web version). This will suffice for the intended model theory application. Also, we may waive the completeness of the tree, but usually still like to have many branches. 

We intend  to deal with an intermediate one (and with the  weakly compact cardinal) in a sequel \cite{Sh:F2064}.

\begin{definition}\label{a28}
    1) Let $ \mathbf{T}_{\fl} $ be the class of structures $\cT$ such that:  

    \begin{enumerate}
        \item[(a)] $\cT = (u,<_*,E,<,\cap,S,R_0,R_1) = (u_{\cT},<^*_{\cT},E_{\cT},<_{\cT},\cap_{\cT},S_{\cT},R^0_{\cT},R^1_{\cT})$ but we may write $s \in \cT$ instead of $s \in u$,
        
        \item[(b)]  $(u,<_*)$ is  a well ordering, so  linear, $u$ non-empty,
        
        \item[(c)]  $<_{\cT}$ is a partial order included in $<_*$, 
        
        \item[(d)]  $(u,<_{\cT})$ is a tree, i.e. if $t \in \cT$ \then\, $\{s:s <_{\cT} t\}$ is well 
        ordered by $<_{\cT}$; 
        the level of $t$ is the order type of this set; the tree is with $\hit(\cT)$ levels, 

Also the tree is normal; that is: if $ t_1, t_2 \in \cT $ 
and $ \{ s: s \le _\cT t_1 \} = 
\{ s: s \le _\cT t_2 \}$  and 
this set has no $ < _\cT $-last element 
then $ t_1=t_2$. 
        
        \item[(e)]   $E$ is an equivalence relation on $u$, convex under $<_*$,   
        
        \item[(f)] 
        
        \begin{enumerate} 
            \item[$(\alpha)$]  each $E$-equivalence class is the set of $t \in \cT$ of level $\varp$ for some $\varp$, so the set of $E$-equivalence classes is naturally well ordered,    
              
            \item[$( \beta )$] we denote the $ \varepsilon$-th equivalence class  by $\cT_{[\varp]}$,  
                
            \item[$( \gamma ) $] $ E $ has  no last $E$-equivalence class if not said otherwise,
            
            \item[$( \delta ) $] let $ \lev_ {\mathscr T } (s  ) = \lev (s, {\mathscr T})$ be $ \varepsilon $ when $ s \in {\mathscr T } _{[\varepsilon ]}$,  equivalently $ \{t \colon t < _ {\mathscr T} s  \} $ has order type $ \varepsilon  $ under the order $<_{\mathscr T},$
              
            \item[$(\varepsilon)$] so $ \hit ({\mathscr T } )$ is $ \bigcup \{\lev_{\cT} ( s ) + 1 \colon s \in {\mathscr T }\}$ and it is a limit ordinal if not said otherwise.  
        \end{enumerate} 
        
        \item[(g)]  if $s \in u,\lev_{\cT}(s) < \zeta < \hit(\cT)$ then there is $t \in \cT_{[\zeta]}$ which is $<_{\cT}$-above $s$, 
        
        \item[(h)]  each $s \in \cT$ has exactly two immediate successors by $<_{\cT}$, 
        
        \item[(i)]  for $s \in \cT$ we let: 
        
        \begin{enumerate}
        \item[$\bullet_1$]   $\cT_{\ge s} = \{t \in \cT:s \le_{\cT} t\}$, 
        
        \item[$\bullet_2$]   $\suc_{\cT}(s) = \{t:t \in \cT_{[\lev(s)+1]}$ satisfies $s <_{\cT} t\}$, 
        \end{enumerate}     
        
        \item[(j)] let $s = t \rest \varepsilon $ mean that $ \lev_{\mathscr T }  (s ) = \varepsilon \le \lev _ {\mathscr T }(t) \wedge (s \le _ {\mathscr T } t ) $, 
        
        \item[(k)]  for $t_1,t_2 \in \cT,t_1 \cap_{\cT} t_2$ is the maximal common lower bound of $t_1,t_2$ (under $\leq_{\cT}$) so we demand it always exists, i.e. $(\cT,<)$ is normal, 
        
        \item[(l)]  for $ {\ell} = 0,1 $ we have  $R_\ell \subseteq \{(s,t) \colon s \in \cT$ and
        $s <_{\cT} t\}$ ,
        and if $ s <_ {\mathscr T } s_1 <_{\mathscr T } s_2$, 
        then $ s R _ {\ell} s_1  $  iff
        $ s R _ {\ell} s_2  $,
        
        \item[(m)] 
        if $s \in  \cT$ then for some 
        $t_0 \ne t_1$ we have $\suc_{\cT}(s) = \{t_0,t_1\}$ and ${\ell} < 2 \Rightarrow  (\forall t)(s R_\ell t$ \Iff \, $t_\ell \le_{\cT} t)$; so $s R_\ell t$ is the 
         analog 
        to $\eta \char 94 \langle \ell \rangle \trianglelefteq \nu$; we may think of  $ \{t\colon s R_ {\ell} t  \}  $ as a division to the left side and the right side of the set of the $t'$s above $ s $.
    \end{enumerate}

    1A) We define $ \mathbf{T} _{\rm wk }$ similarly (``wk'' stands for ``weak''), but we omit clauses (h), (m), replacing them with: 

    \begin{enumerate} 
        \item[(h)'] is $s \in \cT$ and $\lev_{\cT}(s) + \omega \leq \mathrm{ht}(\cT)$, \underline{then} there is $t \in \mathrm{spl}(\cT)$ such that $s \leq_{\cT} t$, where $\mathrm{spl}(\cT) = \{ t \in \cT \colon t \text{ has two immediate successors} \}$, and
    
        \item[(m)']  for $ s \in {\mathscr T } $  either
        $ \suc_{\mathscr T } (s)$  may be  as in part (1) or is empty or a singleton.
    \end{enumerate} 

    1B) For $\cT \in \bfT, <_{\rm{lex}} \, \coloneqq  \, <_{\cT}^{\rm{lex}}$ is the lexicographic order, i.e., 
    $$
    \eta <_{\lex} \nu \ \text{\underline{iff}} \ (\exists \rho)(\rho \, R_{0} \, \nu \wedge \rho \, R_{1} \, \eta) \text{ or } (\eta <_{\cT} \nu \wedge \eta R_{1} \nu) \text{ or } (\nu <_{\cT} \eta \wedge \nu R_{0} \eta).
    $$    
    
    2) Let $\bfT _{\theta,\kappa} = \{\cT \in \bfT \colon$   the tree $\cT$ has $\delta$ levels, for some ordinal $\delta$  of cofinality $\kappa$ and for every $\varp < \delta $ we have $ \theta > |\{s \in \cT \colon s$ of level $\le \varp\}|\}$. 

    3)  Let  $\cT_1 \subseteq_{\fl}  \cT_2$ mean:  
    \begin{enumerate}
        \item[(a)]  $\cT_1,\cT_2 \in\bfT $,  
        
        \item[(b)]  $<_{\cT_1} \, \coloneqq \, <_{\cT_2} \rest u_{\cT_1}$, 
        
        \item[(c)]  if $ {\mathscr T }_1 \models \lqq  \eta \cap \nu = \rho"$ then $ {\mathscr T}_2 \models \text{`}\text{`}  \eta \cap \nu = \rho "$, 
        
        \item[(d)]  $R_{\cT_1,\ell} = R_{\cT_2,\ell} \rest u_{\cT_1}$ for $\ell=0,1$,
        
        \item[(e)] $<^*_{\cT_1} \, \coloneqq \,  <^*_{\cT_2} \rest u_{\cT_1}$; 
        
        \item[(f)] $E_{{\mathscr T } _1} \coloneqq  E_ {{\mathscr T }_2}\rest u_{{\mathscr T } _1}$,  
    \end{enumerate}

    3A) Let $\cT_{1} \subseteq_{\mathrm{wk}} \cT_{2}$ is defined similarly omitting clause (f). 
    
    4) For $s \in  {\mathscr T }$ and $\ell \in \{0,1 \} $,   let $\suc_{\cT,\ell}(s)$ be the unique immediate successor of $s$ in $\cT$ such that $(s,t) \in R^{\cT}_\ell$. 
    
    5) We say $\cT_1,\cT_2 \in \bfT$ are \emph{neighbors} \when \, they are equal except that for each $t \in \cT_{1}$ we can change the order  $<_{\cT_{1}}^{\ast} \, {\rest} \, (t / E_{\cT_{1}})$ to $<_{\cT_{2}}^{\ast} (t / E_{\cT_{2}})$. 
\end{definition}

\begin{definition}\label{a31}
    1) We say $f$ is a $\subseteq$-embedding of $\cT_1 \in \bfT$ into $\cT_2 \in \bfT$ \when: $f$ is an isomorphism from  $ {\mathscr T } _ 1 $ onto ${\mathscr T } _1 ' $ where $ {\mathscr T } _1 ' \subseteq {\mathscr T } _2 $. 

    1A) We say that $f$ is a \emph{semi embedding} of $\cT_{1} \in \bfT$ into $\cT_{2} \in \bfT$, \underline{when} $f$ is a $\subseteq$-embedding of $\cT_{1}$ into some neighbour $\cT_{2}'$ of $\cT_{2}$.

    2) For any ordinal $\alpha$ (limit, if not said otherwise) and sequence $\bar < = \langle <_\beta\colon\beta < \alpha\rangle, $ with $ <_\beta$ a well ordering of ${}^\beta 2$ we define $\cT = \cT_{\alpha,\bar <}$ as follows (omitting $ \bar{ < }$ means \lqq for some"):  
    
    \begin{enumerate}
        \item[(a)]  universe ${}^{\alpha >}2$, 
        
        \item[(b)]  $<_{\cT}$ is $\triangleleft \rest {}^{\alpha >}2$, 
        
        \item[(c)]  $E_{\cT} \coloneqq \{(\eta,\nu) \colon \eta,\nu \in {}^{\beta}2$ for some $\beta < \alpha\}$,  
        
        \item[(d)]  $<^*_{\cT} \coloneqq \{(\eta,\nu) \colon \eta,\nu \in {}^{\alpha >}2$ and $\ell g(\eta) < \ell g(\nu)$ or $(\exists \beta < \alpha)(\ell g(\eta) = \beta = \ell g(\nu) \wedge \eta <_\beta \nu)\}$, 
        
        \item[(e)]  $ R_{\ell} \coloneqq \{ (\eta, \nu ) \colon \eta {\char 94} \langle {\ell} \rangle \trianglelefteq \nu \in {\mathscr T }   \} $,
        	
        \item[(f)] 	$\eta \cap_ {\mathscr T }  \nu \coloneqq \eta \cap \nu $. 
 
    \end{enumerate}
    
    3) For $\cT \in \bfT$ and $\zeta < \rm{ht}(\cT),$ let $<_{\cT, \zeta}$ be $<_{\cT}^{\ast} \rest \cT_{[\zeta]}.$ 

     4)  For any $ \cT \in \mathbf{T} $, let

     \begin{enumerate} 
        \item[(a)] $ { \rm sub}_1( {\mathscr T }) \coloneqq \{ {\mathscr U } \subseteq {\mathscr T } : \cU \not= 
       \emptyset$  and  for every $ s_1 \in {\mathscr U }, s_2 \in {\mathscr T } $   there is  $ t \in {\mathscr U } $ such that $s_1 < _ {\mathscr T } t \wedge \lev _{\mathscr T}(s_2)  \le \lev _ {\mathscr T } (t ) \}$, 

\item[(b)] $  \sub_2(\cT)$ is the set of $ {\mathscr U } \in \sub_1( {\mathscr T } )$  such that: if $ s_1, s_2 \in {\mathscr U } ,
\lev_{\cT}(s_1 ) = \lev_{\cT}(s_2 )$  and 
$ \vert \suc_ {\mathscr T } (s_1) \cap {\mathscr U } \vert =
  2 = \vert \suc_{\mathscr T } (s_2) \cap \cU \vert$ 
    \underline{then} $ s_1= s_2$ . 
\end{enumerate} 

    5) We say $ {\mathscr U } \subseteq {\mathscr T } $  is
    complete (in $ {\mathscr T })$  \underline{when}  
    if $t \in {\mathscr T }_{[\delta ]}, \delta < \mathrm{ht}(\cT)$ is 
    a limit ordinal and 
    $ \{ s : s < t \} \subseteq {\mathscr U }$,  then $ t \in {\mathscr U } $. We say that $\cT$ is complete when $u_{\cT}$ is complete. 
\end{definition}

\begin{claim}\label{a34}
    1) 
    If $\theta = \sup\{(2^{|\alpha|})^+:\alpha < \kappa\}$ and $\bar < = \langle <_\beta:\beta
    < \kappa\rangle$ as in \ref{a31}(2) above, \then\,  
    $\cT_{\kappa,\bar <}$ is well defined and 
    belongs to $\bfT _{\theta,\kappa}.$  

    2) 
    If $ \kappa = \kappa ^{< \kappa }, {\mathscr T } = {\mathscr T } _{\kappa, \bar{<}}$  are as in part (1), 
    \underline{then} there is $ {\mathscr U } \in \sub_{2}(\cT) $
     which is complete in $  {\mathscr T }$  and for every 
     $ s \in {\mathscr U }  $ there exists $ t \in {\mathscr T } $ 
     such that $ s < _ {\mathscr T } t, \lev_ {\mathscr T } (t)
     < \lev _ {\mathscr T } (s) + 2^{\lev_{\mathscr T } (s)}$  
     and $ | \suc _ {\mathscr T } (t) \cap {\mathscr U } | = 2$.  

     It follows that $  {\mathscr U } $  and ${\mathscr T } $ have the same number of $ \kappa$-branches. 
     Note however that,  
     $ {\mathscr T } \upharpoonright \cU \notin  \mathbf{T}$ because clause (m) of \ref{a28}(1) may fail still we may demand $\cT \rest \cU \in \bfT_{\mathrm{wk}}$ (see \ref{a28}(1A)).

     3) 
     For every $ {\mathscr T } _1 \in \mathbf{T} $   
     such that $ \delta =  \hit( {\mathscr T }{_1})$  satisfying $\alpha < \delta   \Rightarrow  \alpha \alpha < \delta$,  there are $ {\mathscr T}_{\mathrm{wk}} \in \mathbf{T} $  
      and $ h $ such that 
      $ \hit( {\mathscr T } _1) = \hit ({\mathscr T } )  $ and 
      $  {\mathscr U } \in \sub_2( {\mathscr T } ) $ and 
      $ h $ is an $\subseteq_{\mathrm{wk}}$-embedding of $ {\mathscr T } _1 $ 
      into $ {\mathscr T } $ with range $ {\mathscr U } $  
      and $ | {\mathscr U } |= | {\mathscr T } _1| =
      | {\mathscr T } | $.
      
\end{claim}

\begin{PROOF}{\ref{a34}}
    It is clear.
\end{PROOF}

\begin{definition}\label{a37}
    1) For $\cT \in \bfT$ let $\eseq_{n} (  \cT ) $ be the set of sequences  $ \bar{ a } $ such that:

    \begin{enumerate} 
        \item[(a)] $\bar a$ is an $<^*_{\cT}$-increasing sequence of length $ n $  of members of $\cT,$
        
        \item[(b)] $k < \ell < n \Rightarrow a_k \cap a_\ell \in \{a_m:m < n\}$; 
       (in fact, $\lqq\in \{ 
       a_m: m \le k \}"). $  
          
        \item[(c)] $ k,{\ell} < n \wedge \lev ( a_k) \le \lev (a_{\ell})  \Rightarrow a_ {\ell} \rest \lev(a_ k )\in \{a_m\colon m < n  \} $,  
    \end{enumerate} 

    1A) For $ {\mathscr U } \subseteq {\mathscr T }$ we let  $\eseq _n(\cU,\cT)$ be $\eseq _n(\cT) \cap ({}^n \cU)$, similarly for part (2).
     
    2) Let $\eseq  (\cT) = \eseq _{< \omega }(\cT) = \cup\{\eseq_n(\cT)\colon n < \omega\}$. 
     
    2A)  For finite $ A \subseteq {\mathscr T } $  we define the sequence $\bar{ b } = \cl(A) = \cl_{\mathscr T } (A) = \cl(A, {\mathscr T})$ as the unique $ \bar{b}$  such that:
    
    \begin{enumerate}
        \item[(a)] $ \bar{ b } \in \eseq ({\mathscr T })$,
        
        \item[(b)]  $ A \subseteq\Rang(\bar{b}),$  
        
        \item[(c)]  $ \Rang ( \bar{ b }) $ is minimal under those restrictions.
    \end{enumerate} 

    Also let  $\pos(A) = \pos(A, {\mathscr T })$ 
    be the unique function $ h $ from $ A $ into $\lg( \bar{ b })$ such that for every $a \in A$ we have: $i = h(a)$  \Iff \, $ b_i = a$. 

    2B) We may replace above $ A $ by a finite sequence $ \bar{a} $, and let  $c\ell_{\mathscr T }( \bar{a})$ be $c\ell_{\cT}(\rang(\bar{a}))$ and $\pos(\bar{a})$ be the function mapping $\ell < \lg(\bar{a})$ to $k$ iff ${a}_{\ell} = b_{k}$.     

    3) We say $\bar a, \bar b \in \eseq(\cT)$ are  $\cT$-similar or $\bar a \sim_{\cT} \bar b$ \when \, for some $n$ we have:
    
    \begin{enumerate}
        \item[(a)]  $\bar a,\bar b \in \eseq_n(\cT)$, 
        
        \item[(b)]  for any $k,i, 
        m < n$ we have (notice that, $<_{\cT}^{\ast}$ is not mentioned):
        
        \begin{enumerate}   
            \item[$\bullet_1$]  $a_k \le_{\cT} a_i$ \Iff \, $b_k \le_{\cT} b_i,$    
            
            \item[$\bullet_2$] $(a_k,a_i) \in R^{\cT}_{\ell} $ \Iff \, $(b_k,b_i) \in
            R^{\cT}_{\ell} $ for ${\ell}  = 0,1$, 
            
            \item[$\bullet_3$]  $a_k \cap_{\cT} a_\ell = a_m$ \Iff \, $b_k \cap_{\cT} b_\ell = b_m$, actually follows, 
            
            \item[$\bullet_4$] $ a_ k = a_ {\ell} \upharpoonright \lev(a_m)$ \underline{iff}  $ b_k = b_{\ell} \rest \lev(b_m)$,
             
            \item[$\bullet_5$] $(a_ k \cap a_m) R_{{\mathscr {\mathscr T}}, {\ell} } a_i$  \underline{iff} $(b_ k \cap b_ m ) R_{ {\mathscr T}, {\ell} } b_i$  for $ {\ell} = 0,1 $; actually follows, 
            
            \item[$\bullet_6$]  $\lev_{\cT}(a_k) \le \lev_{\cT}(a_\ell)$ \Iff \, $\lev_{\cT}(b_k) \le \lev_ {\mathscr T } (b_\ell);$ actually follows, 
            
            \item[$\bullet_7$] $a_k <^*_{\cT} a_{i}$ iff $b_k <^*_{\cT} b_{i}$. 
        \end{enumerate} 
    \end{enumerate}
    
    3A) We say that $\bar{a}, \bar{b} \in {}^{n}\cT$ are \emph{$\cT$-similar} when $\bar{a}' = \cl(\bar{a}), \bar{b}' = \cl(\bar{b})$ are $\cT$-similar and $a_{\ell}' = a_{k} \Leftrightarrow b_{\ell}' = b_{k}$ for any $\ell < \lg(\bar{a}'), k < n.$ 
        	
    4) For $ \bar{ a }  \in {}^{ n } {\mathscr T }$, let $\Lev( \bar{ a } ) $ be the set  $ \{ \lev _{\mathscr T} (a_{\ell} )\colon \ell < n  \} $.  

    5)  We say that $ {{\mathscr T }} \in \mathbf{T} _\exiota $ is weakly $ {\aleph_0} $-saturated \when:

    \begin{enumerate} 
        \item[$(*)$]   for every  $ \varepsilon <   \hit({\mathscr T })$ and  $ s_0, \dots , s_{n-1}$  from $ {\mathscr T } _{[\varepsilon ]},$ there are  $ \zeta \in (\varepsilon, \, \hit({\mathscr T }))$ and   $ t_0 <  ^*_ {\mathscr T } \dots <  ^*_{{\mathscr T } }t_{n-1}$  from $ {\mathscr T } _{[\zeta ]}$ satisfying $ k < n \Rightarrow s_ k < _ {\mathscr T } t_ k $,  
    \end{enumerate} 

    6) For $ {\mathscr T } \in \mathbf{T} _ \exiota $  let:
    
    \begin{enumerate} 
        \item[(a)] $ \icr_n ( {\mathscr T } ) $  be the set of  $ < ^*_ {\mathscr T } $-increasing $ \bar{ a } \in {}^{ n } {\mathscr T } $ and let ${}^{ n } {\mathscr T } $ be the set of sequences of length $ n $ from $ {\mathscr T } $, 
            
        \item[(b)] $ \icr( {\mathscr T } ) = \cup \{\icr_n ({\mathscr T})\colon n < \omega  \} $ and $ \seqq ({\mathscr T } ) = \cup \{ {}^{n} \cT \colon n < \omega  \} $. 
    \end{enumerate} 

    7) For $ {\mathscr T } \in \mathbf{T}_\exiota $:
    
    \begin{enumerate} 
        \item[(a)] for $ \bar{ t } \in \icr({\mathscr T }) $ or just $\seq(\cT)$, let $ \stp (\bar{ t }, {\mathscr T } )  $ be the pair (the similarity type of  $ \cl(\bar{ t } , {\mathscr T }), \pos ( \bar{t}, {\mathscr T })$),  that is all the information from part (3) (of \ref{a37}) and $ \pos$, 
            
        \item[(b)] if in addition, $ {\mathscr U } \subseteq {\mathscr T}$ then we let $ \stp(\bar{ t } , {\mathscr U }, {\mathscr T } )$ be the function mapping $ \bar{ s } \in {}^{ \omega > } {\mathscr U }$ to $ \stp( \bar{ t }  {\char 94} \bar{ s}, {\mathscr T})$.
    \end{enumerate}     

    8) Let $ \mathbb{S}^n$  be the set of similarity types of sequences of length $ n $  in some  $ {\mathscr T } \in \mathbf{T}$,  so the sequences are not necessarily increasing.  
    
    9) Naturally $ \mathbb{S} = \cup \{\mathbb{S}^n\colon n < \omega  \}.$ 

    10) For $ {\mathscr T } \in \mathbf{T} $ and $ n < \omega $ we define
    $ {\rm fseq}_n( {\mathscr T } ) $  as the set of sequences
    $ \bar{ s } = \langle  s_ \eta : \eta \in 
    \cup \{ {}^{ m  }2 : m \le n  \}    \rangle $  
    such that  for some ordinals $ \alpha (0) < \dots 
       \alpha ( n ) $ we hav: 
   \begin{enumerate} 
\item[$ \bullet _1 $] $ s_ \eta \in {\mathscr T } $  is of level $ \alpha (  \lg ( \eta ))$, 
\item[$ \bullet _2 $] $ s_ \eta  < _ {\mathscr T }  s_ \nu $
  when   $ \eta $  is a proper  initial segment of $ \nu $, 
\item[$ \bullet _3 $]   for $ \iota = 0,1 $  above 
  $ s_ \eta R^ {\mathscr T } _ \iota s_ \nu $   iff 
     $ \nu ( \lg (\eta ))  \iota $.
   \end{enumerate} 

   11) The similarity type of $ \bar{ s } \in {\rm fseq}({\mathscr T } ) $
      is the following linear order on 
        $  \cup \{ {}^{ m  }2 : m \le n  \}$: 

        $ \eta < _{\bar{ s }}   \nu $  iff 
        $ \lg( \eta )    < \lg (\nu )$  or 
    $ \lg( \eta )    < \lg (\nu )$  and $ \eta < *_{\mathscr T } \nu $.

So th enumber of similarity types on members of
 $ {\rm fseq}_n ({\mathscr T } )$
   is $ \Pi _{m \le n } (2^m !) $, and we can express our partition relations using
   $ {\rm fseq}_n ({\mathscr T } )$  
   
   instead 
     $ {\rm eseq}_n ({\mathscr T } )$.
\end{definition}

Now comes the main property.

\begin{definition}\label{a40}
    1) For $\cT_1,\cT_2 \in \bfT_{\fl}$ and $n < \omega $ and a cardinal $\sigma$ let $\cT_2 \rightarrow_{\fl} (\cT_1)^n_\sigma$ mean:
    
    \begin{enumerate}
        \item[$(*)$]  if $\bfc \colon \eseq_n(\cT_2) \rightarrow \sigma$, \then \, there is a $\subseteq_{\fl}$-embedding $g$ of $\cT _1$ into $\cT_2$ such that the colouring $\bfc \circ g$ is homogeneous for $\cT_{1},$  which means:
        
        \begin{itemize}
            \item  if $\bar a,\bar b \in \eseq_n(\cT_1)$ are $\cT_2$-similar, then $\bfc(g(\bar a)) = \bfc(g(\bar b))$.
        \end{itemize}
    \end{enumerate}

    2) For $\cT_1,\cT_2 \in \bfT, \, k < \omega$ and $\sigma,$ let $\cT_2 \rightarrow_{\fl} (\cT_1)^{\eend(k)}_\sigma$ mean that:
    
    \begin{enumerate}
        \item[$(*)$]   if $\bfc \colon \eseq (\cT_{2}) \rightarrow \sigma$ \then  \,  there is an $\subseteq_{\fl}$-embedding $g$ of $\cT_1$ into $\cT_2$ such that the colouring $\bfc' = \bfc \circ g$ (see below) satisfies $\bfc'(\bar\eta)$ does not depend  on the last $ k $ levels, that is:  
          
        \begin{enumerate} 
            \item[$\bullet_1$]  the meaning of  $\bfc' = \bfc \circ g$ is that for every $\bar{s} \in \eseq({\mathscr T } _1) $ we have $\mathbf{c} '(\bar{s}) = \mathbf{c} (\langle g(\bar{s}) \rangle)$,
             	
            \item[$\bullet_2$]  if $n < \omega$ and $\bar{a}, \bar{ b } \in \eseq_n (  {\mathscr T } _2 )$ are $ {\mathscr T }_2$-similar
            and  
            $$ {\ell}  < n \wedge (k \le |\Lev( \bar{ a } ) \setminus  \lev (a_{\ell}  )|) \Rightarrow b_ {\ell} = a_ {\ell},
            $$ \underline{then} $\mathbf{c}' ( \bar{ a } ) = \mathbf{c}' (\bar{ b } )$. 
        \end{enumerate} 
    \end{enumerate}

    3) Let $\cT_{2} \to_{\fl} (\cT_{1})_{\sigma}^{\End(k, m)}$ 
    be 
    defined as in part (2), but we restrict in $\bullet_{2}$ demanding that $n \leq m$,  
    so the length of the relevant sequence $ \bar{ a } \in 
    \eseq( {\mathscr T } _2) $ is bounded. 

    4) We define $ {\mathscr T } _1 \rightarrow_{\fl}'  ({\mathscr T } _1)^n_ \sigma $   as in part (1), but $\bfc \colon {}^{n}(\cT_{2}) \to \sigma$ and  $\bar{ a }, \bar{ b } \in {}^{n}(\cT_{2})$. 
    
    5) We define similarly $\cT_{1} \to_{\fl} (\cT_{2})^{\leq n}_ \sigma $ 
    and $\cT_{1} \to_{\fl}' (\cT_{2})^{\leq n} _ \sigma .$ 

    6) We may replace $\bfT_{\fl}$, $\subseteq_{\fl}$, $\to_{\fl}$ by $\bfT_{\mathrm{wk}}$, $\subseteq_{\mathrm{wk}}, \to_{\mathrm{wk}}$ in parts (1)-(5) of Definition \ref{a40}. 
\end{definition}

\begin{remark}\label{g41}
    1) We may  mention some implications among the $ \rightarrow,$
    
    2) Of course, the equality $\bfc(g(\bar{a})) = \bfc(g(\bar{b}))$ is required only if $\bar{a}$ and $\bar{b}$ are $\cT_{2}$-similar since this is the best possible homogeneity, as one can define a coloring according to similarity types. 
\end{remark}

\begin{claim}\label{g43}
    Let $\cT \in \bfT.$  
    
    1) If $A \subseteq \cT$ is finite non-empty with $m$ elements \then:
    
    \begin{enumerate}
        \item[(a)]  For some $ n \le (2m-1)m^{2} $ and $ \bar{ a } \in \eseq_n ({\mathscr T } )$ we have $ A \subseteq \Rang(\bar{ a } )$; moreover $ \max \{\lev_{\mathscr T } (a)\colon a \in A  \} = \max \{\lev_{\mathscr T } (a_{\ell} )\colon{\ell} < n  \};$ in fact $\bar{a} = \cl_{\cT}(A),$
        
        \item[(b)]  If  $ {\mathscr T }  \in \mathbf{T} _ \exiota $ and $ A \subseteq {\mathscr T }$ is finite, then $\cl(A, {\mathscr T}), \pos(A, {\mathscr T})$ are well defined. 
    \end{enumerate}

    2) The number of quantifier free complete $n$-types realized in some $\cT \in \bfT_\exiota $ by some $\bar a \in \eseq_n(\cT)$ is, e.g. $\le 2^{2n^2+n}$ but $\ge n$.

    3)  If $\cT \in \bfT$ is weakly $\aleph_0$-saturated \then \, $\cT$ realizes all possible such types, i.e. each type is realized in some $\cT' \in \bfT;$ here ``$\rm{ht}(\cT)$ is a limit ordinal'' follows. 

    4) Assume ${\mathscr T } \in \mathbf{T}_{\mathrm{wk}}$ and $ {\mathscr U } $  is a subset of $ {\mathscr T } $ closed under $ < _ {\mathscr T } $, (that is $ s <_ {\mathscr T } t \in {\mathscr U } \Rightarrow  s \in {\mathscr U } $).  Let $\lev(  {\mathscr U }, {\mathscr T })  =  \sup  \{  \lev(s,  {\mathscr T }) + 1 \colon s \in {\mathscr U} \}$.

    If $\lev ({\mathscr U }, {\mathscr T } ) \leq \lev (t, {\mathscr T } ) $  and $ A \subseteq {\mathscr U } $  is finite  \then \, $ \bar{ b } = \cl(A \cup \{ t \} ) $ has the form $  \bar{ c } {\char 94} \langle t \rangle $ with $\bar{ c } \in \eseq({\mathscr T }) \cap {}^{ \omega > } {\mathscr U }.$

    5) 
     If $ n < \omega , {\mathscr T } , {\mathscr U } $  
    are as in \ref{a34}(2) \underline{then}  the number  
    $ k^*_n$
    of quantifier free complete $ n $-types realized 
    in $ {\mathscr T } $
             by sequences $  \bar{ a } \in \mathbf{I} _n $ satisfies the following $ k^*_{0} = 1= k^*_1$  and $ k^*_{n + 1} = n k^*_n (n!) $  for $ n \ge 1 $, where

             $ \mathbf{I} _n = \mathbf{I} _n ({\mathscr T }) =
             \{  \bar{ a } \in {}^{ n }{\mathscr U } \colon \bar{a} \text{ is without repetitions and } 
             \langle \lev(a_i): i < \lg ( \bar{ a })  \rangle 
             \text{ is constant}\}. $

    6) If $\cT \in \bfT_{\mathrm{wk}}$, $\cT \in \sub_{2}(\cT)$ and $\bar{a} \in \eseq_{n}(\cT)$, \underline{then} for every $\varp < \mathrm{ht}(\cT)$ for at most one $\ell < \lg(\bar{a})$ we have $\bigwedge_{ \ell = 0}^{1} (\exists w)(a R_{\ell}^{\cT} a_{\ell})$ and\redq{.....}. 
\end{claim}

\begin{PROOF}{\ref{g43}}
    Clearly, (3) and (4) hold, and we shall use them freely.
    
    1) Let $B_1 =\{\eta \cap_{\cT} \nu\colon\eta,\nu \in A\}$ and note that $\eta \in A \Rightarrow \eta = \eta \cap \eta \in B_1$.  Now by induction on $|A|$ easily $|B_1| \le 2m-1$. Let $B_{2} := \{ \eta \rest \lev_{\cT}(\nu)\colon\eta \in A, \nu \in B_{1}$ and $\lev_{\cT}(\eta) \geq \lev_{\cT}(\nu) \}.$ 
    
    Easily $B_{2} = \cl(A, \cT),$ also $\vert B_{2} \vert \leq m^{2} \vert B_{1} \vert = m^{2}(2m - 1).$ 
    

    
    
         
         
         
     
     
     
     
    We may improve the bound\footnote{
    in fact the exact bound is: 
    \begin{itemize}
        \item $\cl_{\cT}(A) = m + (m-1) + {m \choose 2} + {m-1 \choose 2}$ and it is obtained. 
    \end{itemize}

    [Why is this bound?  For any such $A$, define the sets $A[0] \coloneqq A$, $A[1] \coloneqq \{ \eta \cap \nu \colon \eta \neq \nu \in A \}$, $A[2] \coloneqq \{ \eta \rest \lg(\nu) \colon \eta, \nu \in A \text{ and } \lg(\nu) < \lg(\eta) \}$ and 
    $$
    A[3] \coloneqq \{ \eta \rest \lg(\nu \cap \rho) \colon \eta, \nu \in A \text{ and } \eta \ntriangleleft \nu, \eta \ntriangleleft \rho \text{ and } \nu \neq \rho \}.
    $$
    Early $\cl_{\cT}(A) = A[0] \cup A[1] \cup A[2] \cup A[3]$; disjoint and we prove the inequality bu induction on $m$. For $m = 1$, $\cl(A) = A$, so it is clear. If $\vert A \vert = m = n +1$, let $a_{0} <_{I}^{\ast} \dots <_{I}^{\ast} a_{m}$ list $A$ and let $B \coloneqq A \setminus \{ a_{m} \}.$ Easily, $B[i] \subseteq A[i]$ for $i < 4$ and: 
    \begin{itemize}
        \item $A[0] \setminus B[0]$ has at most one element, 

        \item $A[1] \setminus B[1]$ has at most one element, 

        \item $A[2] \setminus B[2]$ has at most $m$ elements, 

        \item $A[3] \setminus B[3]$ has at most $m - 1$ elements.
    \end{itemize}
    Together we are done.]} but this does not matter here; similarly below. 
            
     
    2) Considering the class of such pairs $(\bar a,\cT)$, (fixing $n$); the number of possible $E_{\bar a} = \{(k,i):a_k E_{\cT} a_i\}$ is $\le
    2^{n^2}$ and the number of $<_{\bar a} = \{(k,i):a_k <_{\cT} a_i\}$ is $\le 2^{n^2}$ and the number of $\{(a_k,a_i):(a_k,a_i) \in R_1^{\cT}$ and for no  $j,a_k <_{\cT}
    a_j <_{\cT} a_{i}\}$ is $\le 2^n$.
    
    Lastly, from those we can compute $\{(a_k,a_i):(a_k,a_i) \in
    R^{\cT}_0\}$ as $\{(a_k,a_i)\colon(a_k \cap_{\cT} a_i = a_k) \wedge ((a_k, a_i) \notin R_{1}^{\cT}) \wedge a_{k} \neq \ a_{i} \}$, so together the number is $\le 2^{2n^2+n}$.
    
    Clearly, we can get a better
    bound, e.g. letting $m^\bullet_n(\cT) = |\{\tp_{\qf}(\bar a \rest
    n,\emptyset,\cT):\bar a \in \eseq(\cT)$ has length $\ge n\}|$ then:
    
    \begin{enumerate}
        \item[$(\ast)_{1}$]
        
        \begin{itemize}
            \item[$\bullet_{1}$]  $m^\bullet_n(\cT) = 1$ for $n=0,1,$
        
            \item[$\bullet_{2}$]   $m^\bullet_{n+1}(\cT) \le 4n(m^\bullet _n(\cT)),$
        
            \item[$\bullet_{3}$]   hence $m^\bullet _n(\cT) \le 4^{n-1}(n-1)!.$
        \end{itemize}
    \end{enumerate}
    
    [Why? e.g. for $\bullet_{2}$ notice that $\tp_{\qf}(\bar{a} \rest (n+1), \emptyset, \cT)$ is determined by $q = \tp_{\qf}(\bar{a} \rest n, \emptyset, \cT)$ and the unique triple $(m, \iota, \ell) \in n \times 2 \times 2$ such that:
    
    \begin{enumerate}
        \item[$(\ast)_{1.1}$] 
        
        \begin{enumerate}
            \item[(a)] $m < n$ is such that $\lev(a_{m} \cap a_{n})$ is maximal, hence $a_{m} <_{\cT} a_{n},$ 
            
            \item[(b)] $a_{m} R_{\iota} a_{n},$

            \item[(c)] $\ell= 0$ iff $\Lev_{\cT}(a_{n}) > \Lev_{\cT}(a_{n-1}).$
        \end{enumerate}
    \end{enumerate}
    
    As there are $\leq 4n$ possibilities, we are done.]
    
    It suffices to consider the case $\cT$ is weakly $\aleph_{0}$-saturated (see \ref{a37}(5), \ref{g43}(3)) and then we can get exact values.
    
    Now for $n \geq k \geq 1$ let, $$m_{n, k}^{\ast}(\cT) :=  \vert \{ \tp_{\qf}(\bar{a}, \emptyset, \cT)\colon\bar{a} \in \eseq_{n}(\cT) \ \text{such that} \ \vert \{ \ell\colon\lev(a_{\ell}) = \max( \Lev(\bar{a})) \} \vert = k \} \vert.$$
    
    So, 
    
    \begin{itemize}
        \item $m_{1, 1}^{\ast}(\cT) = 1, \, m_{1, 0}^{\ast}(\cT) = 0$ and stipulate $m_{0, k}^{\ast}(\cT) = 0,$
        
        \item if $n = k \geq 1,$ then $m_{n, k}^{\ast}(\cT) = 1,$
        
        \item if $2k - 1 > n \geq k \geq 1,$ then $m_{n, k}^{\ast}(\cT) = 0,$
        
        \item if $n \geq 1,$ then $m_{n+1, 1}^{\ast}(\cT) = \Sigma \{ 2k \cdot m_{n, k}^{\ast}(\cT)\colon k \in [1, n] \},$
    \end{itemize}
    
    and more generally, 
    
    \begin{itemize}
        \item if $n > k \geq 1,$ then $$m_{n+k, k}^{\ast} = \sum \biggl \{ \ell! \cdot  {\ell \choose \ell_{1}} \cdot {\ell - \ell_{1} \choose \ell_{2}} \cdot 2^{\ell} \cdot m_{n, \ell}^{\ast}(\cT)\colon\ell, \ell_{0}, \ell_{1}, \ell_{2} \in [0, n), \, \ell = \ell_{0} + \ell_{1} + \ell_{2} \biggr \}.$$ 
    \end{itemize}
    
    [Why? Considering $p = \tp_{\qf}(\bar{a} \rest (n + k), \emptyset, \cT)$ we fix $q = \tp(\bar{a} \rest n, \emptyset, \cT),$ let $\ell$ be maximal such that $ n - \ell \leq i < n \Rightarrow \lev(a_{i}) = \lev(a_{n - 1})$ (equivalently $\lev(a_{n - \ell})) = \lev(a_{n - 1})).$ For $\iota = 0, 1, 2,$ let $S_{\iota} = \{m \colon n - \ell \leq m < n$ and $\iota = \vert \{ j < k\colon a_{m} <_{\cT} a_{j} \} \vert\},$ so $(S_{0}, S_{1}, S_{2})$ is a partition of $[n - \ell, n).$ Let $S_{1}^{\bullet} = \{ m \in S_{1} \colon $ if $j < k$ then $a_{m} R_{1} a_{s}\}.$ Fixing $\ell$ the number of possibles $q$'s is $m_{n, \ell}^{\ast}(\cT)$ and fixing $q$ (and so $\ell$) the freedom left is choosing $\ell_{0}, \ell_{1}, \ell_{2} \geq 0$ such that $\ell_{1} + 2 \ell_{2} = k$ and then choosing the partition $(S_{0}, S_{1}, S_{2})$ which have ${\ell \choose \ell_{1}} {\ell - \ell_{2} \choose \ell_{2}}$ possibilities we have $2^{\ell_{1}}$ possible choices of $S_{1}^{\bullet}$ and lastly $k$ possible linear orders of $\{ a_{i}\colon i \in [n, n+k) \}$ clearly we are done.] 
    
    3), 4), 5)  Clear.   
\end{PROOF}

\begin{claim}\label{a74}
    Let\footnote{If $\sigma < \aleph_{0}$ we have parallel results depending on  decreasing $\sigma$ in the 
    conclusion 
    on the bounds from  \ref{a34},  
    that is, for part (2): if $\cT \to  (\cT)_{\sigma(1)}^{\End(1, m(1))},$ then $\cT \to (\cT)_{\sigma(2)}^{\End(n, m(2))}.$} 
    $\sigma \geq \aleph_{0}$ be a cardinal and $\cT \in \bfT$. Then: 

    1) If $\cT \to (\cT)_{\sigma}^{\rm{end}(1)}$ \underline{then} $\cT \to (\cT)_{\sigma}^{\rm{end}(k)}$ for every $k < \omega.$ 
    
    2) If $\cT \to (\cT)_{\sigma}^{\rm{end}(1)}$ \underline{then} $\cT \to (\cT)_{\sigma}^{n}$ for every $n < \omega.$ 
    
    3) If $k \geq 1$ and $\cT_{\ell} \in \bfT$ for $\ell = 0, \dots, k$ and $\cT_{\ell + 1} \to (\cT_{\ell})_{\sigma}^{\rm{end}(1, m)}$ for $\ell < k,$ \underline{then} $\cT_{k} \to (\cT_{0})_{\sigma}^{\rm{end}(k, m)},$ hence $\cT_{k} \to (\cT_{0})_{\sigma}^{\leq m}.$
    
    
    
    
    
\end{claim}

\begin{proof} 
    Clear. 
\end{proof}

\subsection{Forcing in $\ZFC$}\label{1B}

\begin{remark}\label{f2}
    Concerning the choice of $m(\ast)$ in $\ref{f5}$ below (given $\bfm$), it is minor from the author's point of view, i.e., its value is immaterial for  the model theoretic results. 
    
    Trivially $m(\ast) = 2m$ suffices. 
\end{remark} 

\begin{definition}\label{f3}
    Let $\cT \in \bfT$ and $\bar{s} \in \eseq(\cT)$. 

    (1) Let $\mathrm{last-lev}(\bar{s}) \coloneqq \{ \ell < \lg(\bar{s}) \colon$ if $k < \lg(\bar{s})$, then $\lev_{\cT}(s_{k}) \leq \lev_{\cT}(s_{\ell}) \}$ and let $\mathrm{last-ele}(\bar{s}) \coloneqq \{ s_{\ell} \colon \ell \in \mathrm{last-lev}(\bar{s}) \}$ and $\mathrm{init-lev}(\bar{s}) = \{ \ell < \lg(\bar{s}) \colon \ell \notin \rm{last-lev}(\bar{s}) \}$. 

    (2) Let $\mathrm{order}(\bar{s}, \cT) \coloneqq \{ \sqsubset \colon \sqsubset$ is a linear order on $\mathrm{last-ele}(\bar{s})\}$. 

    (3) Let $\mathrm{Eseq}(\cT)$ be the set of pairs $(\bar{s}, \sqsubset)$, where $\bar{s} \in \eseq(\cT)$ and $\sqsubset \in \mathrm{order}(\bar{s}, \cT)$. 

    (4) The similarity type of $(\bar{s}, \sqsubset) \in \mathrm{Eseq(\cT)}$ is naturally defined: ($(\bar{s}_{1}, \sqsubset_{1}), (\bar{s}_{2}, \sqsubset_{2}) \in \mathrm{Eseq}(\cT)$) have the same similarity type \underline{iff} $\mathrm{sim-tp}(\bar{s}_{1}, \cT_{1}) = \mathrm{sim-tp}(\bar{s}_{2}, \cT_{2}')$ \underline{when} for $\iota = 1, 2$, we let: 

    \begin{itemize}
        \item[$\bullet_{1}$] $\alpha_{\iota}$ be such that $\{ s_{\iota, \ell} \colon \ell \in \mathrm{last-ter}(\bar{s}_{\iota}, \cT_{\iota}) \} \subseteq \cT_{[\alpha_{\iota}]}$, 

        \item[$\bullet_{2}$] $\cT_{\iota}'$ be like $\cT_{\iota}$ except that we change $<_{\cT_{\iota}, \alpha_{\iota}}$ for every $\ell, k \in \mathrm{last-term}(\bar{s}_{1})$ we have $s_{\iota, \ell} <_{\cT_{\iota}'} s_{\iota, k} \Leftrightarrow \ell \sqsubset_{\iota} k$. 
    \end{itemize}

    (So it does not follows that $\bar{s}_{1}$, $\bar{s}_{2}$ have the same similarity type).

    (5) For $\cT_{1}$, $\cT_{2} \in \bfT$, we have $\cF_{2} \to^{1} (\cF_{1})_{\sigma}^{\mathrm{end}(1, m)}$ \underline{when} (A) $\Rightarrow$ (B), where: 

    \begin{enumerate}
        \item[(A)] $\bfc \colon \mathrm{Eseq}(\cT_{2}) \to \sigma$. 

        \item[(B)] There is a semi embedding $h$ from $\cT_{1}$ into $\cT_{2}$ (see Definition \ref{a31}(1A)) such that: 

        \begin{itemize}
            \item if $(\bar{s}_{1}, \sqsubset_{1})$, $(\bar{s}_{2}, \sqsubset_{2}) \in \mathrm{Eseq}_{\leq m}(\cT_{2})$, $\vert \mathrm{last-lev}(\bar{s}_{1}) \vert = m = \vert \mathrm{last-lev}(\bar{s}_{2}) \vert$ and $\bar{s}_{2} \rest \mathrm{init-lev}(\bar{s}_{1}) = \bar{s}_{2} \rest \mathrm{init-lev}(\bar{s}_{2})$, \underline{then} we have $\bfc(h(\bar{s}_{1}), \sqsubset_{1}) = \bfc(\bar{s}_{2}, \sqsubset_{2})$. 
        \end{itemize}
    \end{enumerate}

    (6) We define $\cF_{2} \to^{2} (\cF_{1})_{\sigma}^{\mathrm{end}(1, m)}$ similarly, but in (B) we demand $h$ to be an embedding.

    (7) We define $\cF_{2} \to^{\ell} (\cF_{1})_{\sigma}^{\mathrm{end}(k, m)}$ for $\ell = 1, 2$ and $k \in [1, \omega)$ similarly. 
\end{definition}

\begin{definition}\label{f3b}
    1) Fixing $\bar{m} = \langle m_{l} \colon l \leq l(\ast) \rangle$, $m_{\ell} \geq 1$ and $\cT \in \bfT$, let $\eseq_{\bar{m}}(\cT)$ is the set of $\bar{s} \in \eseq(\cT)$ such that: 

    \begin{enumerate}
        \item the set $\{ \lg(s_{i}) \colon i < \lg(\bar{s}) \}$ has $\ell(\ast) + 1$ members,

        \item let $\alpha_{0} < \dots < \alpha_{\ell(\ast)}$ list it,

        \item $\vert \{ \ell \colon \lg(s_{\ell}) = \alpha_{\ell} \} \vert = m_{\ell}$ for $\ell \leq \ell(\ast)$.
    \end{enumerate}

    2) For $\cT_{1}, \cT_{2} \in \bfT$, let $\cT_{2} \to^{\ell} (\cT_{1})_{\sigma}^{\mathrm{end}(1, \bar{m})}$, mean (A)$\Rightarrow$(B), where: 

     \begin{enumerate}
        \item[(A)] $\bfc \colon \mathrm{Eseq}_{\bar{m}}(\cT_{2}) \to \sigma$. 

        \item[(B)] There is a semi embedding $h$ if $\ell = 1$ and an embedding if $\ell = 2$ from $\cT_{1}$ into $\cT_{2}$ such that: 

        \begin{itemize}
            \item if $(\bar{s}_{1}, \sqsubset_{1})$, $(\bar{s}_{2}, \sqsubset_{2}) \in \mathrm{Eseq}_{\leq m}(\cT_{2})$, $\vert \mathrm{last-lev}(\bar{s}_{1}) \vert = m_{\ell(\ast)} = \vert \mathrm{last-lev}(\bar{s}_{2}) \vert$ and $\bar{s}_{2} \rest \mathrm{init-lev}(\bar{s}_{1}) = \bar{s}_{2} \rest \mathrm{init-lev}(\bar{s}_{2})$, \underline{then} we have $\bfc(h(\bar{s}_{1}), \sqsubset_{1}) = \bfc(\bar{s}_{2}, \sqsubset_{2})$. 
        \end{itemize}
    \end{enumerate}

    3) For $\cT_{1}$, $\cT_{2} \in \bfT$, we have $\cF_{2} \to^{\ell} (\cF_{1})_{\sigma}^{\bar{m}}$ \underline{when} (A) $\Rightarrow$ (B), where: 

    \begin{enumerate}
        \item[(A)] $\bfc \colon \mathrm{Eseq}_{\bar{m}}(\cT_{2}) \to \sigma$. 

        \item[(B)] There is a semi embedding $h$ if $\ell_{1}$, and an embedding if $\ell_{2}$ from $\cT_{1}$ into $\cT_{2}$ such that: 

        \begin{itemize}
            \item if $\bar{s}_{1}$, $\bar{s}_{2} \in \mathrm{Eseq}_{\bar{m}}(\cT_{2})$, \underline{then} we have $\bfc(h(\bar{s}_{1})) = \bfc(\bar{s}_{2})$. 
        \end{itemize}
    \end{enumerate}
\end{definition}

\begin{claim}\label{f4}
    The obvious implications hold. In particular, if $\bar{m} = \langle m_{\ell} \colon \ell \leq \ell(\ast) \rangle$ and $\cT_{\ell + 1} \to^{2} (\cT_{\ell})_{\sigma}^{\mathrm{end}(1, \bar{m} \rest (\ell + 1))}$ for $\ell < \ell(\ast)$, then $\cT_{\ell(\ast)} \to^{2} (\cT_{0})_{\sigma}^{\bar{m}}$. 
\end{claim}

\begin{majorclaim}\label{f5}
    In $\bfV^{\bbP}$ we have $\cT_{2} \to^{2} (\cT_{1})_{\sigma}^{\rm{end}(1, \bfm)}$ \underline{when} (see Remark \ref{f7}) for a suitable $m(\ast)$: 
    
    \begin{enumerate}
        \item[(a)]   we have: 

        \begin{itemize}
            \item[($\bullet _1$)] $\kappa = \kappa^{< \kappa}$, 
        
            \item[($\bullet _2$)] $ \lambda \to (\kappa^{+})_{\Upsilon}^{m(\ast)},$ where $ \Upsilon = 2^{\kappa}$, $\lambda$ regular, really $\lambda = (\beth_{m(\ast) - 1}(\kappa^{+}))^{+},$ 

            \item[$(\bullet_{3})$] $\sigma < \kappa$ and $\sigma \geq \aleph_{0}$. 
        \end{itemize}

        \item[(b)] $\bbP = \Cohen(\kappa, \lambda),$ 
        
        \item[(c)] $\cT_{2} \in \bfT$ expands $({}^{\kappa(+) >}2, \lhd)$ in $\bfV^{\bbP},$
        
        \item[(d)] In $\bfV^{\bbP},$ $\cT_{1} \in \bfT_{\kappa, \kappa}$ and $\cT_{1} \subseteq \cT_{1}^{+},$ where $\cT_{1}^{+}$ expands $({}^{\kappa >}2, \lhd)$ and  so $\otp(\cT_{1, [\alpha]}, <_{\cT_{1, \alpha}}) < \kappa$ for $\alpha < \kappa,$ see \ref{a31}(3),
    \end{enumerate}
\end{majorclaim}

\begin{remark}\label{f7}
    1) As in \cite{Sh:289}, $m(\ast) = 2m$ suffice. More on the value of $m(\ast)$, see \cite{Sh:1258}  (see more in \cite{Sh:F2064} and \cite{Sh:F1523}). 

    2) Can we  for $\to_{\mathrm{wk}}^{2}$ get a better to $m(\ast)$? Intend to return to this in \cite{Sh:F2064} and \cite{Sh:F1523}.   
\end{remark}

\begin{PROOF}{\ref{f5}}
    First, 
    
    \begin{enumerate}
        \item[$\boxdot_{1}$] Without loss of generality, $\cT_{1} \in \bfV$ is an object (not just  a $\bbP$-name).  
    \end{enumerate}
    
    [Why? Let $\name{\cT}_{1}$ be a $\bbP$-name, then for some $u \in [\lambda]^{\leq \kappa}$ we have $\name{\cT}_{1}$ is a $\bbP_{u}$-name. We can force by $\bbP_{u},$ so as $\bbP / \bbP_{u} = \Cohen(\kappa, \lambda),$ we are done.]
    
    So $\kappa, \cT_{1}, \name{\cT}_{2}$ are well defined ($\name{\cT}_{2}$ a $\bbP$-name). Let $\name{\bfc}$ be a $\bbP$-name, $\name{\bfc}\colon \mathrm{Eseq}(\name{\cT}_{2}) \to \sigma,$ without loss of generality  be such that: 
    
    \begin{enumerate}
        \item[$\boxdot_{2}$] if $(\bar{t}, \sqsubset) \in \mathrm{Eseq}(\cT_{2}),$ then from $\name{\bfc}(\bar{t}, \sqsubset)$ we can compute:

        \begin{enumerate}
            \item[(a)]$\name{\bfc}(\bar{t}, \sqsubset')$ when $\sqsubset' \in \rm{order}(\bar{t}, \cT)$, 

            \item[(b)] the similarity type of $\bar{t}$ in $\cT_{2}$,

            \item[(c)] the similarity type of $\name{\bfc}((\bar{t}, \sqsubset) \rest u)$ when $u \subseteq \dom(\bar{t}), \bar{t} \rest u \in \mathrm{Eseq}(\name{\cT}_{2}).$
        \end{enumerate} 
    \end{enumerate}
    
    Let  $\name{\bar{\eta}} = \langle \name{\eta}_{\alpha}\colon \alpha < \lambda \rangle$ be the generic of $\bbP,$  so $\Vdash$``$\name{\eta}_{\alpha} \in {}^{\kappa}2$'' and let $\name{\bar{\eta}}_{u} = \langle \name{\eta}_{\alpha}\colon \alpha \in u \rangle$ for $u \subseteq \lambda.$ 
    
    Next, in $\bfV,$ we choose: 
    
    \begin{enumerate}
        \item[$(*)_{1}$]
        
        \begin{enumerate}
            \item[(a)] let $\chi > \lambda$ and $<_{\chi}^{\ast}$ a well ordering of $\cH(\chi),$
            
            \item[(b)] let $\mathfrak{B} \prec \mathfrak{A}_{0} = (\cH(\chi), \in, <_{\chi}^{\ast})$ be of cardinality $\kappa$ such that $[\mathfrak{B}]^{< \kappa} \subseteq \mathfrak{B}$ and $\lambda, \kappa, \mu, \sigma, \cT_{1}, \name{\cT}_{2}, \name{\bfc} \in \mathfrak{B};$
            
            \item[(c)] let $u_{\ast} = \mathfrak{B} \cap \lambda \in [\lambda]^{\kappa},$
            
            \item[(d)] let $\bfG_{u_{\ast}} \subseteq \bbP_{u_{\ast}}$ be generic over $\bfV_{0} = \bfV,$ and let $\bfG \subseteq \bbP$ be generic over $\bfV_{0}$ such that $\bfG_{u_{\ast}} \subseteq \bfG,$
            
            \item[(e)] let $\bar{\eta}_{u} = \langle \name{\eta}_{\alpha}[\bfG]\colon \alpha \in u \rangle,$ for $u \subseteq \lambda,$
            
            \item[(f)] let $\bfV_{1} = \bfV_{0}[\bfG_{u_{\ast}}] = \bfV_{0}[\bar{\eta}_{u_{\ast}}],$
            
            \item[(g)] let $\bfV_{2} = \bfV[G] = \bfV_{0}[\bar{\eta}_{\lambda}] = \bfV_{1}[\bar{\eta}_{\lambda \setminus u_{\ast}}].$ 
        \end{enumerate}
    \end{enumerate}
    
    \begin{enumerate}
        \item[$(\ast)_{2}$] 
        
        \begin{enumerate}
            \item[(a)] let $\name{\cT}_{0}$ be the $\bbP$-name of the sub-structure of $\name{\cT}_{2}$ with set of elements $\{ \name{\eta}\colon \name{\eta}$ is a canonical $\bbP$-name of a member of $\name{\cT}_{2}$ and this name belongs to $\mathfrak{B} \},$
            
            \item[(b)] let $\delta_{\ast} = \delta(\ast)$ be  $\kappa^{+} = \min(\kappa^{+} \setminus u_{\ast}) = \kappa^{+} \cap u_{\ast},$ 
             noting $\delta_{\ast}$ has cofinality $\kappa$ because $u_{\ast} = \mathfrak{B} \cap \kappa^{+}, [\mathfrak{B}]^{< \kappa} \subseteq \mathfrak{B}$ and $\Vert \mathfrak{B} \Vert = \kappa,$ 
            
            \item[(c)] let $\langle \delta_{\varp}\colon\varp < \kappa \rangle$ be increasing continuous with limit $\delta_{\ast}$ in $\bfV_{0},$
            
            \item[(d)] $\gB_{2} = \gB[\bfG_{\lambda}], \, \gB_{1} = \gB_{2} \rest \{ \name{\tau}[\bfG_{\lambda}]\colon\name{\tau}$ is a $\bbP_{u}$-name from $\gB$ for some  $u \in [\lambda]^{\leq \kappa} \cap \gB \},$ $ \gB_{0} = \gB,$ so $\gB_{2} \prec \mathfrak{A}_{2} = \cH(\chi)[\bfG_{\lambda}]$ and $\gB_{1} \cap {}^{\kappa \geq} \lambda = \gB_{2} \cap ({}^{\kappa \geq} \lambda)^{\bfV_{1}}.$  
        \end{enumerate}
    \end{enumerate}
    
    Clearly, 
    
    \begin{enumerate}
        \item[$(\ast)_{3}$] 
        
        \begin{enumerate}
            \item[(a)] $\Vdash_{\bbP_{\lambda}}$``$\name{\cT}_{0} \subseteq \name{\cT}_{2}$ is closed under initial segments, is of cardinality $\kappa$ and has $\delta_{\ast}$ levels and is closed under unions of increasing chains of length $< \kappa$ and $\nu \in \name{\cT}_{0} \Rightarrow \nu^{\smallfrown} \langle 0 \rangle, \, \nu^{\smallfrown} \langle 1 \rangle \in \name{\cT}_{0},$  and $\alpha < \delta_{\ast} \Rightarrow (\forall \nu \in \cT_{0})(\exists \rho)[\nu \lhd \rho \in \cT_{0} \wedge \lg(\rho) \geq \alpha],$ so $\name{\cT}_{0} \in \bfT$'',
            
            \item[(b)] $\name{\cT}_{0}$ is actually a $\bbP_{u_{\ast}}$-name and we can use $\delta_{\ast} = \mathfrak{B} \cap \kappa^{+}$ as its set of levels. 
        \end{enumerate}
    \end{enumerate}
    
    \begin{enumerate}
        \item[$(\ast)_{4}$] 
        
        \begin{enumerate}
            \item[(a)] let $\cT_{0} = \name{\cT}_{0}[\bfG_{u_{\ast}}], \, \bfc_{0} = \bfc \rest \eseq(\cT_{0})$ so they are from $\bfV_{1},$
            
            \item[(b)] let $\bbP_{\ast} = \bbP / \bfG_{u_{\ast}} = \bbP_{\lambda \setminus u_{\ast}}.$ 
        \end{enumerate}
    \end{enumerate}
    
    \begin{enumerate}
        \item[$(\ast)_{5}$]
        
        \begin{enumerate}
            \item[(a)] for each $\alpha \in \lambda \setminus u_{\ast},$ in $\bfV_{1}[\eta_{\alpha}]$ there is $\eta_{\alpha}^{\bullet} \in \lim_{\delta_{\ast}}(\cT_{0})^{\bfV_{1}[\eta_{\alpha}]}$ hence $\lg(\eta_{\alpha}^{\bullet}) = \delta_{\ast}$ such that $\varp < \delta_{\ast} \Rightarrow \eta_{\alpha}^{\bullet} \rest \varp \in \cT_{0}$ and $\eta_{\alpha}^{\bullet}$ is a generic $\delta_{\ast}$-branch of $\cT_{0}$ over $\bfV_{1},$ 
            
            \item[(b)] Clearly $\lim_{\delta_{\ast}}(\cT_{0})^{\bfV_{1}[\eta_{\alpha}]} \subseteq \cT_{2}[ \mathbf{G} 
            _{\lambda}].$ 
        \end{enumerate}
    \end{enumerate}
    
    We shall work in $\bfV_{1}.$ 
    
    \begin{enumerate}
        \item[$(\ast)_{6}$] (in $\bfV_{1}$) for $\alpha \in \lambda \setminus u_{\ast},$ let $\name{\eta}_{\alpha}^{\bullet}$ be a $\bbP_{\{ \alpha \}}$-name of $\eta_{\alpha}^{\bullet},$ so without loss of generality for some $\kappa$-Borel function $\bfB\colon{}^{\kappa} 2 \to {}^{\delta(\ast)} 2,$ from $\bfV_{1}$ we have $\Vdash$``$\name{\eta}_{\alpha}^{\bullet} = \bfB(\name{\eta}_{\alpha})$ is as above in $(\ast)_{5}$(a)'';  
 note that $\bfB$ does not depend of $ \alpha$.
    \end{enumerate}

    [Why? See the proof of $ (\ast)_7$]
    \begin{enumerate}
        \item[$(\ast)_{7}$] Without loss of generality, 
        
        \begin{enumerate}
            \item[(a)] Recall we had in $(\ast)_{2}$(c) an increasing sequence $\langle \delta(\varp) \coloneqq \delta_{\varp} \colon \varp < \kappa \rangle$ so that $\delta(\ast) = \bigcup_{\varp < \kappa}\delta(\varp)$
        
            \item[(b)]  Let $\langle \bfB_{\varp}\colon\varp < \kappa \rangle$ be such that $\bfB_{\varp}$ is a $\kappa$-Borel function from ${}^{\kappa}2$ to ${}^{\delta(\varp)}2$  from $\bfV_{1}$ and $\Vdash$``$\name{\eta}_{\alpha}^{\bullet} \rest \delta_{\varp} = \bfB_{\varp}(\name{\eta}_{\alpha})$'' for $\alpha \in \lambda \setminus u_{\ast}$,
            
            \item[(c)] In $\bfV_{0}$, let $\name{\bfB}, \name{\bfB}_{\varp}$ be $\bbP_{u_{\ast}}$-names forced to be as above, can be considered as $\bbP_{u_{\ast}}$-name. 
        \end{enumerate}
    \end{enumerate}

    [Why? 

    For $\eta \in {}^{\omega} 2$, define $\nu = \nu_{\eta} \in {}^{\kappa} \kappa$ as follows: for $\varp < \kappa$, we let $\nu_{\eta}(\varp)$ be the the unique $\zeta < \chi$ such that $\name{\eta}_{\alpha}(\zeta) = 1 \wedge \varp = \otp\{ \xi < \zeta \colon \name{\eta}_{\alpha}(\xi) = 1 \}$ if there is such $\zeta$ and $0$ otherwise. Now, $\cT_{0} \cap {}^{\delta(\varp)} \lambda$ has $\leq \kappa$ members, hence we can interprete $\bfB_{\varp}(\eta) = \bfB_{\varp}(\eta \rest (\varp + 1))$, a member of $\cT_{0} \cap {}^{\delta(\varp)} \lambda$, which is $\unlhd$-increasing with $\varp$ and $\bfB(\eta) = \bigcup \{ \bfB_{\varp}(\eta \rest (\varp + 1)) \colon \varp < \kappa \}$. That is, we can first choose $ \mathbf{B}_{\varp}$ 
    by induction on $ \varp < \kappa$ such that $\bfB_{\varp}(\eta)$ depend just on $\nu_{\eta} \rest (\varp + 1)$ and $(\ast)_{8}$ below holds  
    and then choose $\bfB$.]
    
    \begin{enumerate}
        \item[$(\ast)_{8}$] If $\varp < \kappa, \, \nu \in {}^{\kappa} 2, $ and $ \bfB(\nu) = \rho,$ then $\bfB_{\varp}(\nu \rest \zeta) = \rho
    \upharpoonright \delta_{\varp}  
    .$
    \end{enumerate}
    
    Recalling we work in  $\bfV_{1}$ and $\bfV$, $\bfV_{1}$ have the same cardinal arithmetic, there are $\cU, \bar{N}$ such that: 
    
    \begin{enumerate}
        \item[($\ast)_{9}$]  
        
        \begin{enumerate}
            \item[(a)] $\cU \subseteq \lambda \setminus u_{\ast},$
            
            \item[(b)] $\otp(\cU)$ 
               is $\kappa^{+}$,
            
            \item[(c)] $\bar{N} = \langle N_{u}\colon u \in [\cU]^{\leq \bfm} \rangle,$
            
            \item[(d)] $N_{u} \cap N_{v} \subseteq N_{u \cap v}$ when $u, v \in [\cU]^{\leq \bfm}$,
            
            \item[(e)] $\kappa, \cT_{1}, \name{\cT}_{2}, \name{\bfc}, \cB, \cB_{1} \in N_{u} \prec \mathfrak{A}_{0}, \Vert N_{u} \Vert = \kappa, [N_{u}]^{< \kappa} \subseteq N_{u}$ for $u \in [\cU]^{\leq \bfm}$,
            
            \item[(f)] if $u, v \in [\cU]^{\leq m}$ and $ \vert u \vert = \vert v \vert,$ then there is a unique isomorphism $ \mathbf{g} 
            _{u, v}$ from $N_{v}$ onto $N_{u}$ 
               and it 
            is the identity on $(\kappa+1) \cup \{ \bfc, \name{\eta}, \bfB \}$ hence on $\cT_{2}, \langle \name{\eta}_{\alpha}, \name{\eta}_{\alpha}^{\bullet}\colon\alpha < \lambda \setminus u_{\ast} \rangle$ and maps $v$ onto $u,$ 
              and if $ v_ 1 \subseteq v, u_1 = \mathbf{g}_{u, v} [v_1] \subseteq u$  then 
              $  \mathbf{g} _{u_1, v_{1}} \upharpoonright N_v \subseteq
              \mathbf{g} _{u,v}$,
            
            \item[(g)] $\langle \name{\eta}^{\bullet}_{\alpha}\colon\alpha \in \cU \rangle$ is $<_{\name{\cT}_{2}}^{\ast}$-increasing. 
        \end{enumerate}
    \end{enumerate}
    
    [Why?   
        By clause (a)($ \bullet _3$)  of the assumption of 
        \ref{f5}  as 
    in \cite{Sh:289},  there is such $\bfm(\ast)$ (see \ref{f7}), for clause (g) recall that $<_{\name{\cT}_{2}}^{\ast}$ is a (linear) well-ordering.] 
    
    \begin{enumerate}
        \item[$(\ast)_{10}$] \underline{notation}:
        
        \begin{enumerate}
            \item[(a)] for finite $u \subseteq \cT_{0, [\varp]}$ for some $\varp < \delta_{\ast},$ let

            \begin{enumerate} 
\item[($ \alpha $)]                
$\bfH_{u} := \{h\colon h$ is a one-to-one function from $u$ into $\cU \},$

\item[($ \beta  $)]  We let $ \mathbf{F} _u$ be the set of $ f $ such that for some 
   $(\bar{s}, \sqsubset) = (\bar{s}_{f}, \sqsubseteq_{f})$  and $ g= g_f $ we have
      \begin{enumerate} 
            \item[($ \bullet _1$)] $ (\bar{ s}, \sqsubset) \in \mathrm{Eseq}({\mathscr T } _1)$, 
            
            \item[($ \bullet _2$)] $ \mathrm{last-ele}[ \bar{ s }] = v$, and note $\bar{s} = \cl_ {\mathscr{T}_{1}} ( \{ s_i : i < \lg(\bar{s}) \} )$ so we let $ u= u_f$, 

            \item[($ \bullet _3$)] $ g $ is an increasing function from $ \{ \lev_ {\mathscr{T}_{1} } ( s_i ): i < \lg( \bar{ s } )\}$ into $\mathrm{ht}(\cT_{0}) = \delta_{\ast}$,
                 
             \item[($ \bullet _4$)]  $f$ is a semi embedding of $ {\mathscr T } _1 \upharpoonright  \{ s_i : i < \lg( \bar{ s } \}$ into $\cT_{0}$ mapping $\{ s_{i} \colon i \in v \}$ onto $u$,
             
              \item[($ \bullet _5$)] so $ f(s_i) \in {\mathscr T } _{0, [g(i)]}$ 
                 where $ \level_{{\mathscr T } _1}(f(s_{i})) = g(\lev_{{\mathscr T } _1}(s_i))$.
      \end{enumerate} 
  
\end{enumerate} 
            \item[(b)] If in $\bfV_{2},$ $(\bar{s}, \sqsubset) \in \mathrm{Eseq}(\cT_{0}), u = \mathrm{last-ele}[\bar{s}] \in [\cT_{0, [\varp]}]^{\leq \bfm 
            }$ and $h \in \bfH_{u},$ then we let $\bar{s}^{[h]}$ be the $(\bar{t}, \sqsubset) \in \mathrm{Eseq}(\cT_{1})$ such that $\lg(\bar{t}) = \lg(\bar{s})$ and $i \in \lg(\bar{s}) \setminus u \Rightarrow t_{i} = s_{i}$ and $i \in  \mathrm{last-lev} (\bar{s}) \Rightarrow t_{i} = s_{i}^{\smallfrown} 
            \name{ \eta }_{i}' $ where $ \name{\eta }_{i}' = 
            \eta_{h(s_{i})}^{\bullet} \rest [ \varp(\bar{s}), \delta_{\ast} ).$ 
        \end{enumerate}
    \end{enumerate}
    
    We define $\AP := \bigcup_{\varp < \kappa} \AP_{\varp},$ where $\AP_{\varp}$ is the set of objects $\bfa$ which consists of (so $\varp = \varp_{\bfa}, \bar{\nu} = \bar{\nu}_{\bfa},$ etc).
    
    \begin{enumerate}
        \item[$\boxplus_{\bfa}^{0}$] 
        
        \begin{enumerate}
            \item[(a)] $\varp < \kappa,$
        
            \item[(b)] 

            \begin{enumerate}
                \item[$\bullet_{1}$] $\bar{\nu} = \langle \nu_{\rho}\colon\rho \in \cT_{1, [\varp]} \rangle$ is with no repetitions;  
            (will serve as an approximation to an embedding), 
                \item[$\bullet_{2}$] $\bff$ is an $\subseteq_{\fl}$-embedding of $\cT_{1, [\leq \varp]} = \cT_{1} \rest \{ \eta \in \cT_{2} \colon \lg(\eta) \leq \varp \}$ into $\cT_{0}$, 

                \item[$\bullet_{3}$] $\bff(\rho) = \nu_{\rho}$ for $\rho \in \cT_{1, [\varp]}$. 
            \end{enumerate}
            
            \item[(c)] $\bar{\eta} = \langle \eta_{\rho}\colon\rho \in \cT_{1, [\varp]} \rangle;$  
            ($\nu_{\rho}$ will serve as condition in $ \mathbb{P} _{\{ \alpha \} }$ for some $ \alpha$), 
            
            \item[(d)] $\nu_{\rho} \in \cT_{0, [\zeta]}$ and $\eta_{\rho} \in {}^{\kappa >}2$ for some $\zeta = \zeta_{\bfa} < \kappa$ 
            (for all $\rho \in \cT_{1, [\varp]}$), 
            
            \item[(e)] $\bfB_{\zeta}(\eta_{\rho}) = \nu_{\rho},$
            
            \item[(f)] $\bar{p} = \langle p_{u, h}\colon u \in [\cT_{0,  [\varp]}]^{\leq 
            \bfm)}, 
            h \in \bfH_{u} \rangle,$ where $p_{u, h} \in \bbP_{\lambda \setminus u_{\ast}},$
            
            \item[(g)] $p_{u, h} \in N_{h[u]}$ and $[\rho \in u \wedge h(\rho) = \alpha \Rightarrow p_{u, h}(\alpha) = \eta_{\rho}],$
            
        \end{enumerate}
    \end{enumerate}
    
    \begin{enumerate}
        \item[$ $]
        
        \begin{enumerate}
            \item[(h)] if $h_{1}, h_{2} \in \bfH_{u}$ and $h_{1}[u] = h_{2}[u]$ \underline{then} $\alpha \in \dom(p_{u, h_{1}}) \setminus h_{1}[u] \Rightarrow p_{u, h_{1}}(\alpha) = p_{u, h_{2}}(\alpha),$
            
            \item[(i)] if $h_{1}, h_{2} \in \bfH_{u}$ and $\rho_{1}, \rho_{2} \in u \Rightarrow
            [h_{1}(\rho_{1}) < h_{1}(\rho_{2}) \equiv h_{2}(\rho_{1}) < h_{2}(\rho_{2})]$,  
            \underline{then} $\mathbf{g} 
             _{h_{2}[u], h_{1}[u]}$ maps $p_{u, h_{1}}$ to $p_{u, h_{2}},$
            
            \item[(j)] if $u_{1}, u_{2} \in [^{\varp} 2]^{\leq \bfm}$, $h_{\ell} \in \bfH_{u_{\ell}}$ for $\ell = 1, 2$, and $h_{1},  h_{2}$ are compatible, then $(p_{u_{1}, h_{1}} \rest N_{h_{1}[u_{2}]})$ and $ p_{u_{2}, h_{2}}$ are compatible.  
            
                 
        \end{enumerate}
    \end{enumerate}

    Let further 
    
    \begin{enumerate}
        \item[$\boxplus_{\bfa}^{1}$] $\AP^{+} = \bigcup \{ \AP_{\varp}^{+}\colon\varp < \kappa \},$ where $\AP_{\varp}^{+}$ is the set of $\bfa \in \AP_{\varp}$ such that:
        
        \begin{enumerate}
            \item[(k)] $p_{u, h}$ forces a value to $\name{\bfc}(\bar{s}, \sqsubset)$ \underline{when} for some $f \in \mathbf{F}_{u}$ we have $\bar{s} = \bar{s}_{f}$, $u_{f} = u$, $\sqsubset = \sqsubset_{f} = \{ (s_{\ell}, s_{k}) \colon s_{\ell}, s_{k} \in u$ and $\{ h(s_{\ell}) < h(s_{k}) \}$.  
        \end{enumerate}
    \end{enumerate}
    
    
    \begin{enumerate}
        \item[$\boxplus_{2}$]  we define the two-place relation $\leq_{\AP}$ as follows: $\bfa_{1} \leq_{\AP} \bfa_{2}$ \underline{iff}: 
        
        \begin{enumerate}
            \item[(a)] $\bfa_{1}, \bfa_{2} \in \AP,$
            
            \item[(b)] $\varp_{1} = \varp_{\bfa_{1}} \leq \varp_{\bfa_{2}} = \varp_{2},$ 
            
            \item[(c)] if $\iota \in \{ 0, 1 \},$ $\rho_{1} \in {}^{\varp(1)}2, \rho_{1}^{\smallfrown} \langle \iota \rangle \unlhd \rho_{2} \in {}^{\varp(2)} 2,$ then $\eta_{\bfa_{1}, \rho_{1}} {}^{\smallfrown} \langle \iota \rangle \unlhd \eta_{\bfa_{2}, \rho_{2}}$ and $\nu_{\bfa_{1}, \rho_{1}} {}^{\smallfrown} \langle \iota \rangle \unlhd \nu_{\bfa_{2}, \rho_{2}},$ 
            
            \item[(d)] if $n \leq m,$ $u_{1} \in [\cT_{1, [\varp(1)]}]^{n} \subseteq [{}^{\varp(1)}2]^{n}, u_{2} \in [\cT_{1, \varp(1)}]^{n}, u_{1} = \{ \rho \rest \varp(1)\colon\rho \in u_{2}  \rangle,$ and $ h_{\ell} \in \bfH_{u_{\ell}}$ for $\ell = 1, 2$,  \underline{then} $\rho \in u_{2} \Rightarrow h_{2}(\rho) \rest \varp(1) = h_{1}(\rho \rest \varp(1)),$ and $p_{u_{1}, h_{1}} \leq_{\bbP_{\ast}} p_{u_{2}, h_{2}}.$  
        \end{enumerate}
    \end{enumerate}
    
    \begin{enumerate}
        \item[$\boxplus_{3}$] $(\AP, <_{\AP})$ is a partial order.  
    \end{enumerate}
    
    [Why? Read the definitions.] 
    
    \begin{enumerate}
        \item[$\boxplus_{4}$] for $\varp = 0$ there is $\bfa \in \AP_{\varp}.$ 
    \end{enumerate}
    
    [Why? Trivial.]
    
    \begin{enumerate}
        \item[$\boxplus_{5}$] 

        \begin{enumerate}
            \item[(a)] If $\varp < \kappa$ is a limit ordinal and $\bfa_{\zeta} \in \AP_{\zeta}$ for $\zeta < \varp$ is $\leq_{\AP}$ increasing, \underline{then} there is $\bfa_{\varp} \in \AP_{\varp}$ such that $\zeta < \varp \Rightarrow \bfa_{\zeta} <_{\AP} \bfa_{\varp}.$ 

            \item[(b)] In clause (a), we can add: if $\bfa_{\zeta} \in \AP^{+}_{\zeta}$ for $\zeta < \varp$ \underline{then} $\bfa_{\varp} \in \AP_{\varp}^{+}$.
        \end{enumerate}

    \end{enumerate}
    
    [Why? Straightforward recalling that $[N_{u}]^{< \kappa} \subseteq N_{u}$.]
    
    \begin{enumerate}
        \item[$\boxplus_{6}$] if $\varp < \kappa$ and $\bfa \in \AP_{\varp}$ there is $\bfb$ such that: 
        
        \begin{enumerate}
            \item[(a)] $\bfb \in \AP_{\varp},$
            
            \item[(b)] $\bfa \leq_{\AP} \bfb,$
            
            \item[(c)] $\varp_{\bfb} = \varp_{\bfa} +1,$ 
            
        \end{enumerate}
    \end{enumerate}
    
    [Why? Straightforward.]
    
    \begin{enumerate}
        \item[$\boxplus_{7}$]  if $\bfa \in \AP_{\varp},$ \underline{then} there is $\bfb \in \AP_{\varp}^{+}$ such that $\bfa \leq_{\AP} \bfb.$
    \end{enumerate}

    Toward this, we first show:  

    \begin{enumerate}
        \item[$\boxplus_{7.1}$] if $\bfa \in \AP_{\varp}$, $u \in [\cT_{1, [\varp]}]^{\leq \bfm}$ and $p, \bar{s}$, $\sqsubset$ are as in clause (k) of $\boxplus_{\bfa}^{1}$ \underline{then}  there is $\bfb \in \AP_{\varp}$ such that $\bfa \leq_{\AP} \bfb$ and $p_{\bfa, u, h}$ forces a value to $\name{\bfc}(\bar{s}, \sqsubset)$.   
    \end{enumerate}

    [Why? First choose $p \in N_{h[u]} \cap \bbP$ above $p_{\bfa, u, h}$ forcing a value to $\bfc(\bar{s}, \sqsubset)$. Then choose $p_{\bfb, u_{1}, h}$ for relevant pairs by combining $p_{\bfa, u_{1}, h_{1}}$ and $p$ (so $p_{\bfb, u, h} = p$) remembering $\boxplus_{2}$.]
    
    \begin{enumerate}
        \item[$\boxplus_{8}$] we can choose $\bfa_{\varp} \in \AP_{\varp}$ by induction on $\varp < \kappa$ such that $\zeta < \varp \Rightarrow \bfa_{\zeta} \leq_{\AP} \bfa_{\varp}$ and $\varp = \zeta + 1 \Rightarrow \bfa_{\varp} \in \AP_{\varp}^{+}.$
    \end{enumerate}
    
    [Why? Use $\boxplus_{4}$ for $\varp = 0,$ use $\boxplus_{5}$ for $\varp$ a limit ordinal and $\boxplus_{6} + \boxplus_{7}$ for $\varp = \zeta +1$.] 
    
    
    Now define $h\colon\cT_{1} \to \cT_{2}$ by $\rho \in \cT_{1, [\varp]} \Rightarrow h(\rho) = \eta_{\bfa_{\varp}, \rho}$ and check.]

    Lastly, 

    \begin{enumerate}
        \item[$\boxplus_{9}$] $\bigcup \{ \bff_{\bfa_{\varp}} \colon \varp < \kappa \}$ is an embedding as is desired.  
    \end{enumerate}
\end{PROOF}

\begin{remark}\label{f8}
    1) Using the end of \S1A we get the desired conclusions.
    
    2) In \ref{f5} we may state and prove the variant with the square bracket. In more details,
    
    \begin{enumerate}
        \item[(A)] We say $\bar{a}, \bar{b} \in \eseq_{m}(\cT)$ are weakly $\cT$-similar as before but omitting ``$a_{\ell} <_{\cT}^{\ast} a_{k} \Leftrightarrow b_{\ell} <_{\cT}^{\ast} b_{k}$''; that is, when $\lg(\bar{a}) = \lg(\bar{b})$ and for some permutation $\pi$ of $\lg(\bar{a})$ for $k, \ell, m < \lg(\bar{a}),$ we have: 
        
        \begin{enumerate}
            \item[(a)] $\lev_{\cT}(a_{\ell}) = \lev_{\cT}(a_{k}) \Rightarrow  \lev_{\cT}(b_{\pi(\ell)}) = \lev_{\cT}(b_{\pi(k)}),$
            
            \item[(b)] $a_{\ell} <_{\cT} a_{k} \Leftrightarrow b_{\pi(\ell)} <_{\cT} b_{\pi(k)}.$
        \end{enumerate}
        
        \item[(B)] We  replace ``$\cT_{2} \to (\cT_{1})_{\sigma}^{\rm{end}(k, m)}$'' by ``$\cT_{2} \to {[\cT_{1}]_{\sigma, j}^{\rm{end}(k, m)}}$'' for suitable finite $j$ which means that \ref{a40}(1)$(\ast)\bullet_{2}$ is replaced by: 
        
        \begin{enumerate}
            \item[$\bullet_{2}$] if $n < \omega$ and $\bar{a} \in \eseq_{n}(\cT_{2}),$ \underline{then} the following set has at most $j$ elements: \newline $\{ \bfc'(\bar{b})\colon\bar{b} \in \eseq_{n}(\cT_{2})$ is weakly $\cT_{1}$-similar to $\bar{a}$ and $\ell < n \wedge (k \leq \vert \Lev(\bar{a})) \setminus \lev(a_{\ell}) \vert) \Rightarrow b_{\ell} = a_{\ell} \}.$
        \end{enumerate}
    \end{enumerate}
    
    3) This will be enough for the model theory, and if we use minimal j (well, depends on $\bar{s} \rest v(\bar{s})$), we get back to \ref{f5}. 
\end{remark}

\begin{conclusion}\label{f11}
    Assume  $ \kappa = \kappa ^{< \kappa }$  and 
    $ \kappa < 
    \chi \leq \infty$ is limit and $\theta \in [\kappa, \chi) \Rightarrow 2^{\theta} = \theta^{+},
    $  
    and if $ \chi $ is singular then $ 2^ \chi = \chi ^+$.

    We can find a forcing notion $\bbP$ such that: 
    \begin{enumerate}
        \item[(a)]  $ \mathbb{P} $ is a $ ( < \kappa ) )$-complete 
        forcing notion of cardinality  $ \chi ^{< \chi }$,
        
        \item[(b)] $\bbP$ collapses no cardinal, changes no cofinality,
        
        \item[(c)] $2^{\theta} < \theta^{+ \omega}$ 
               for $\theta   \in [ \kappa , \chi )$, 
        
        \item[(d)] in $\bfV^{\bbP},$ if $\theta^{+ \omega} \leq \chi,$ 
        then 
        for every $ k, m < \omega $,  
        for some $n < \omega,$ for every $\cT_{1}  \in \mathbf{T} $  
        expanding $({}^{\theta > }2, \lhd)$ there is $\cT_{2} \in \mathbf{T} $  
        expanding $({}^{\theta(+n) >} 2, \lhd)$ 
            we have   
        $\cT_{2} \to (\cT_{1})_{\theta}^{\rm{end}( k,m) } .$ 
    \end{enumerate}
\end{conclusion}


\begin{discussion}\label{f22} 
    We may  like to  replace ${}^{\kappa >}2$ by ${}^{\kappa >}I$ and even use
    creature tree forcing, see Roslanowski-Shelah \cite{Sh:470},  \cite{Sh:777},  Goldstern-Shelah \cite{Sh:808}) but (in second thought, for $ \kappa = {\aleph_0}$ maybe see the paper with Zapletal  \cite{Sh:952}). That is, for $ \kappa > {\aleph_0}$ in each node we have a forcing notion which is quite complete, but of cardinality $ < \kappa $ = set of levels.
    
    So we do not have a tree but a sequence of creatures, $\langle \mathfrak{c} _ \varepsilon \colon \varepsilon < \hit({\mathscr T } )\rangle,$ such that for a colouring  we like to find $ \mathfrak{d}_ \varepsilon  \in  \Sigma (\mathfrak{c} _ \varepsilon )$ for $ \varepsilon < \kappa,$ which induces a sub-tree in which the colouring is $1$-end-homogeneous. 
    Alternatively we have $ \langle \mathfrak{c}_ \eta\colon\eta \in {\mathscr T } \rangle $  where $ \mathfrak{c} _ \eta $ is a creature with set of possible  values being in $\suc_{\mathscr T}(\eta)$, see \cite{Sh:777}.
    
    Clearly the answer  is that we can, but it is not clear how interesting it is. We can just, 
    
    \begin{enumerate}
        \item[$\boxdot$] replace $2$ by $\Upsilon \in [2, \kappa)$ and ${}^{\kappa >} 2$ by ${}^{\kappa >} \Upsilon$; in the Definition \ref{a28} replace $R_{\cT, \ell}$ ($\ell < 2$) by $R_{\cT, \ell}$ $(\ell < \Upsilon)$  and add: if $s \in \cT,$ then $\suc_{\cT}(s)$ is either a singleton or is $\{ s_{\ell}\colon\ell < \Upsilon \},$ where $s_{\ell} R_{\cT, \ell} s$ for $\ell < \Upsilon.$  
    \end{enumerate}
\end{discussion}

\bibliographystyle{amsalpha}
\bibliography{shlhetal}

\newcommand{\etalchar}[1]{$^{#1}$}
\providecommand{\bysame}{\leavevmode\hbox to3em{\hrulefill}\thinspace}
\providecommand{\MR}{\relax\ifhmode\unskip\space\fi MR }
\providecommand{\MRhref}[2]{%
  \href{http://www.ams.org/mathscinet-getitem?mr=#1}{#2}
}
\providecommand{\href}[2]{#2}
\begin{thebibliography}{She78b}

\bibitem[DH17]{DoHa17}
Natasha Dobrinen and Dan Hathaway, \emph{The {H}alpern-{L}\"{a}uchli theorem at
  a measurable cardinal}, J. Symb. Log. \textbf{82} (2017), no.~4, 1560--1575.

\bibitem[DS79]{Sh:85}
Keith~J. Devlin and Saharon Shelah, \emph{{A note on the normal Moore space
  conjecture}}, Canadian J. Math. \textbf{31} (1979), no.~2, 241--251.
  \MR{528801}

\bibitem[DS04]{Sh:692}
Mirna D{\v{z}}amonja and Saharon Shelah, \emph{{On
  $\vartriangleleft^*$-maximality}}, Ann. Pure Appl. Logic \textbf{125} (2004),
  no.~1-3, 119--158, \href{https://arxiv.org/abs/math/0009087}{arXiv:
  math/0009087}. \MR{2033421}

\bibitem[GS05]{Sh:808}
Martin Goldstern and Saharon Shelah, \emph{{Clones from creatures}}, Trans.
  Amer. Math. Soc. \textbf{357} (2005), no.~9, 3525--3551,
  \href{https://arxiv.org/abs/math/0212379}{arXiv: math/0212379}. \MR{2146637}

\bibitem[HL66]{HaLa66}
J.~D. Halpern and H.~L{\"{a}}uchli, \emph{A partition theorem}, Trans. Amer.
  Math. Soc. \textbf{124} (1966), 360--367.

\bibitem[HL71]{HaLe67}
J.~D. Halpern and A.~Levy, \emph{The boolean prime ideal theorem does not imply
  the axiom of choice}, Axiomatic set theory (Providence, R.I.), Proceedings of
  symposia in pure mathematics, 'vol. 13, part 1, American Mathematical
  Society, 1967, 1971, pp.~83--134.

\bibitem[Lav71]{Lv71}
Richard Laver, \emph{On fraiss\'e's order type conjecture}, Annals of
  Mathematics \textbf{93} (1971), 89--111.

\bibitem[Lav73]{Lv73}
\bysame, \emph{An order type decomposition theorem}, Annals of Mathematics
  \textbf{98} (1973), 96--119.

\bibitem[Mil79]{Mil79}
Keith~R. Milliken, \emph{A {R}amsey theorem for trees}, J. Combin. Theory Ser.
  A \textbf{26} (1979), no.~3, 215--237.

\bibitem[Mil81]{Mil81}
\bysame, \emph{A partition theorem for the infinite subtrees of a tree}, Trans.
  Amer. Math. Soc. \textbf{263} (1981), no.~1, 137--148.

\bibitem[RS99]{Sh:470}
Andrzej Ros{\l}anowski and Saharon Shelah, \emph{{Norms on possibilities. I.
  Forcing with trees and creatures}}, Mem. Amer. Math. Soc. \textbf{141}
  (1999), no.~671, xii+167, \href{https://arxiv.org/abs/math/9807172}{arXiv:
  math/9807172}. \MR{1613600}

\bibitem[RS07]{Sh:777}
\bysame, \emph{{Sheva-Sheva-Sheva: large creatures}}, Israel J. Math.
  \textbf{159} (2007), 109--174,
  \href{https://arxiv.org/abs/math/0210205}{arXiv: math/0210205}. \MR{2342475}

\bibitem[S{\etalchar{+}}a]{Sh:F1822}
S.~Shelah et~al., \emph{Tba}, In preparation. Preliminary number: Sh:F1822.

\bibitem[S{\etalchar{+}}b]{Sh:F2060}
\bysame, \emph{Tba}, In preparation. Preliminary number: Sh:F2060.

\bibitem[S{\etalchar{+}}c]{Sh:F2064}
\bysame, \emph{Tba}, In preparation. Preliminary number: Sh:F2064.

\bibitem[S{\etalchar{+}}d]{Sh:F1523}
\bysame, \emph{Tba}, In preparation. Preliminary number: Sh:F1523.

\bibitem[Shea]{Sh:1258}
Saharon Shelah, \emph{{Consistency of square bracket partition relation}}.

\bibitem[Sheb]{Sh:950}
\bysame, \emph{{Dependent dreams: recounting types}},
  \href{https://arxiv.org/abs/1202.5795}{arXiv: 1202.5795}.

\bibitem[Shec]{Sh:E59}
\bysame, \emph{{General non-structure theory and constructing from linear
  orders; to appear in Beyond first order model theory II}},
  \href{https://arxiv.org/abs/1011.3576}{arXiv: 1011.3576} Ch. III of The
  Non-Structure Theory" book [Sh:e].

\bibitem[She71a]{Sh:E17}
\bysame, \emph{{Two cardinal and power like models: compactness and large group
  of automorphisms}}, Notices Amer. Math. Soc. \textbf{18} (1971), no.~2, 425,
  71 T-El5.

\bibitem[She71b]{Sh:8}
\bysame, \emph{{Two cardinal compactness}}, Israel J. Math. \textbf{9} (1971),
  193--198. \MR{0302437}

\bibitem[She75]{Sh:37}
\bysame, \emph{{A two-cardinal theorem}}, Proc. Amer. Math. Soc. \textbf{48}
  (1975), 207--213. \MR{357105}

\bibitem[She76]{Sh:49}
\bysame, \emph{{A two-cardinal theorem and a combinatorial theorem}}, Proc.
  Amer. Math. Soc. \textbf{62} (1976), no.~1, 134--136 (1977). \MR{434800}

\bibitem[She78a]{Sh:74}
\bysame, \emph{{Appendix to: ``Models with second-order properties. II. Trees
  with no undefined branches'' (Ann. Math. Logic 14 (1978), no. 1, 73--87)}},
  Ann. Math. Logic \textbf{14} (1978), 223--226. \MR{506531}

\bibitem[She78b]{Sh:a}
\bysame, \emph{{Classification theory and the number of nonisomorphic models}},
  Studies in Logic and the Foundations of Mathematics, vol.~92, North-Holland
  Publishing Co., Amsterdam-New York, 1978. \MR{513226}

\bibitem[She89]{Sh:289}
\bysame, \emph{{Consistency of positive partition theorems for graphs and
  models}}, {Set theory and its applications (Toronto, ON, 1987)}, Lecture
  Notes in Math., vol. 1401, Springer, Berlin, 1989, pp.~167--193. \MR{1031773}

\bibitem[She92]{Sh:288}
\bysame, \emph{{Strong partition relations below the power set: consistency;
  was Sierpi\'nski right? II}}, {Sets, graphs and numbers (Budapest, 1991)},
  Colloq. Math. Soc. J\'anos Bolyai, vol.~60, North-Holland, Amsterdam, 1992,
  \href{https://arxiv.org/abs/math/9201244}{arXiv: math/9201244}, pp.~637--668.
  \MR{1218224}

\bibitem[She00]{Sh:702}
\bysame, \emph{{On what I do not understand (and have something to say), model
  theory}}, Math. Japon. \textbf{51} (2000), no.~2, 329--377,
  \href{https://arxiv.org/abs/math/9910158}{arXiv: math/9910158}. \MR{1747306}

\bibitem[Sil80]{Sil80}
Jack~H. Silver, \emph{Counting the number of equivalence classes of {B}orel and
  coanalytic equivalence relations}, Ann. Math. Logic \textbf{18} (1980),
  no.~1, 1--28. \MR{568914}

\bibitem[SU19]{Sh:1141}
Saharon Shelah and Danielle Ulrich, \emph{{Torsion-free abelian groups are
  consistently ${\mathrm{a}}\Delta^1_2$-complete}}, Fund. Math. \textbf{247}
  (2019), no.~3, 275--297, \href{https://arxiv.org/abs/1804.08152}{arXiv:
  1804.08152}. \MR{4017015}

\bibitem[SZ11]{Sh:952}
Saharon Shelah and Jind{\v{r}}ich Zapletal, \emph{{Ramsey theorems for product
  of finite sets with submeasures}}, Combinatorica \textbf{31} (2011), no.~2,
  225--244. \MR{2848252}

\bibitem[Vau65]{Va65}
R.~L. Vaught, \emph{A l\"owenheim-skolem theorem for cardinals far apart}, {The
  Theory of Models} (J.~V. Addison, L.~A. Henkin, and A.~Tarski, eds.),
  North--Holland Publishing Company, 1965, pp.~81--89.

\end{thebibliography}

\end{document}